\documentclass [leqno] {amsart}
  
\usepackage {amssymb, amsthm}

\usepackage[all]{xy}
\usepackage {hyperref} 

\theoremstyle{definition}
\newtheorem{defn}{Definition}[section]

\theoremstyle{plain}

\newtheorem{prop}[defn]{Proposition}
\newtheorem{lem}[defn]{Lemma}
\newtheorem{cor}[defn]{Corollary}
\newtheorem{thm}[defn]{Theorem}

\theoremstyle{remark}
\newtheorem{rmk}[defn]{Remark}
\newtheorem{convention}[defn]{Convention}
\newtheorem{no}[defn]{}

\numberwithin{equation}{defn}

\newcommand{\eq}[2]{\begin{equation}\label{#1}#2 \end{equation}}

\newcommand{\mlnl}[1]{\begin{multline*}#1 \end{multline*}}

\newcommand{\arir}{\ar@{^{(}->}}
\newcommand{\aril}{\ar@{_{(}->}}
\newcommand{\are}{\ar@{>>}}

\newcommand{\xr}[1] {\xrightarrow{#1}}
\newcommand{\xl}[1] {\xleftarrow{#1}}

\newcommand{\inj}{\hookrightarrow}

\newcommand{\ul}{\underline}
\newcommand{\ol}{\overline}

\newcommand{\codim}{{\rm codim}}

\newcommand{\Div}{{\rm div}}
\newcommand{\Hom}{{\rm Hom}}

\newcommand{\Spec}{{\rm Spec \,}}

\newcommand{\Nm}{{\rm Nm }}

\newcommand{\id}{{\rm id}}

\newcommand{\Sm}{\mathrm{Sm}}
\newcommand{\Cor}{\mathrm{Cor}}

\newcommand{\Ker}{{\rm Ker}}
\newcommand{\im}{{\rm Im}}

\newcommand{\sC}{{\mathcal C}}

\newcommand{\sL}{{\mathcal L}}

\newcommand{\sO}{{\mathcal O}}

\newcommand{\sV}{{\mathcal V}}

\newcommand{\A}{{\mathbb A}}

\newcommand{\N}{{\mathbb N}}
\renewcommand{\P}{{\mathbb P}}

\newcommand{\Z}{\mathbb{Z}}

\newcommand{\fc}{\mathfrak{c}}

\newcommand{\fm}{\mathfrak{m}}

\newcommand{\fp}{\mathfrak{p}}

\newcommand{\e}{\epsilon}

\begin{document}

\title{Suslin homology of relative curves with modulus}


\author{Kay R\"ulling \and Takao Yamazaki}
\date{\today}
\address{Freie Universit\"at Berlin, Arnimallee 7, 14195 Berlin, Germany}
\email{kay.ruelling@fu-berlin.de}
\address{Mathematical Institute, Tohoku University,
  Aoba, Sendai 980-8578, Japan}
\email{ytakao@math.tohoku.ac.jp}

\begin{abstract}
We compute the Suslin homology of relative curves with modulus.
This result may be regarded as a modulus version of
the computation of motives for curves,
due to Suslin and Voevodsky.
\end{abstract}

\thanks{The first author is supported by the ERC Advanced Grant 226257 and the DFG Heisenberg Grant RU 1412/2-1,
the second author is supported by JSPS KAKENHI Grant (15K04773).}
\maketitle

\tableofcontents

\section{Introduction}

Let $B$ be a variety over a field $k$.
The relative Suslin homology 
$H_i^S(X/B)$ of a $B$-scheme $X$ 
is introduced in \cite[{\S 3}]{SuVo96}
and is computed when $X \to B$ is a relative curve
that admits a good compactification
(\cite[Thm. 3.1]{SuVo96}; see also \cite{Li}).
This result is later interpreted as a computation
of motives of curves in Voevodsky's triangulated category of motives
\cite[Thm. 3.4.2]{VoTmot}.

In the present note we introduce 
Suslin homology $H_i^S(X/B, D)$ with modulus $D$
(see Def. \ref{defn:rel-Suslin-homology}).
Here $B$ is a smooth $k$-scheme, $X \to B$ is a proper $k$-morphism,
and $D$ is an effective Cartier divisor on $X$ 
such that $U:=X \setminus |D|$ is smooth
and equidimensional over $k$.
When $B=\Spec k$, the definition of $H_i^S(X/B, D)$
is due to S. Saito and is studied in \cite{KSY2}.
There is a canonical homomorphism
$H_i^S(X/B, D) \to H_i^S(U/B)$
(see \ref{propertiesRSH} \eqref{propertiesRSH6}),
but it is usually far from being bijective.
Our main result describes $H_i^S(X/B, D)$ when
$X \to B$ is a relative curve
(see Remark \ref{rmk:simplecase}):

\begin{thm}\label{thm-intro}
Let $S$ be a smooth connected $k$-scheme,
$\sC$ an integral normal $k$-scheme,
and $p : \sC \to S$ a proper $k$-morphism.
Suppose that the generic fiber of $p$ is $1$-dimensional,
that $U := \sC \setminus |D|$ is smooth over $k$,
that $p|_U : U \to S$ is affine,
and that $D$ is contained in an affine open neighborhood in $X$.
Then we have
\[H_i^S(\sC/S, D)=
\begin{cases} 
{\rm Pic}(\sC, D), &\text{if } i=0,\\
0, &\text{if } i\ge 1.
\end{cases}
\]
Here ${\rm Pic}(\sC, D)$ denotes the relative Picard group
(see \ref{not:OXD}).
\end{thm}

By comparison with \cite{SuVo96},
we find that the canonical map 
$H_i^S(\sC/S, D) \to H_i^S(U/S)$
is bijective if $D$ is reduced.
Our proof of Theorem \ref{thm-intro} is,
however, totally different from \cite{SuVo96} (or from \cite{Li}).
Indeed, in loc. cit. the starting point of the proof (for reduced $D$)
is the homotopy invariance of ${\rm Pic}(\sC, D)$, 
which is no longer valid if $D$ is not reduced.
Our proof is rather reminiscent of \cite[Thm. 4.3]{BS14}.

Actually we shall show a slightly more general result,
see Theorem \ref{thm:RSHcurves}.
We expect that our result will be interpreted
as a computation of ``motives with modulus'' for curves,
just like the case for ``motives without modulus''.

\subsection*{Acknowledgements}
The authors thank Florian Ivorra for interesting discussions around Theorem \ref{thm-intro} and 
H\a'el\a`ene Esnault for inviting the second author to the FU Berlin 
with her ERC Advanced Grant 226257.
As noted above, the definition of the Suslin homology with modulus
is originally due to Shuji Saito.
We thank him for his permission to use his idea.

\begin{convention}\label{convention}
By convention a $k$-scheme is a scheme which is separated and of finite type over a field $k$.
We denote by ${\rm CDiv(X)}$ the group of Cartier divisors on $X$. If $D$ is an effective Cartier divisor
we denote by $|D|$ its support. If $X$ and $Y$ are $k$-schemes we write $X\times Y=X\times_k Y$. 
\end{convention}

\section{Relative finite correspondences}\label{sec:rel-cor}\setcounter{subsection}{1}

\begin{no}\label{no:Cartier-pb}
We recall a criterion for the pullback of an Cartier divisor to exist. 
Let $X$ be a $k$-scheme and $D$ a Cartier divisor on $X$.
Let $f:Y\to X$ be a morphism of $k$-schemes and assume
that $f({\rm Ass}(\sO_X))\subset X\setminus |D|$,
              where ${\rm Ass}(\sO_X)$ denotes the set of associated points of $X$,
             (e.g. $Y$ is integral and $f(Y)$ is not contained in $|D|$).
Then the pullback $f^*D$ exists as a Cartier divisor.
(Indeed, we may assume that $D$ is effective and then it suffices to show that the pullback 
of a local equation of $D$ to $Y$ is  a regular element. This is a local question
so let $y\in Y$ be a point, $x=f(y)$ and $d\in\sO_{X,x}$ be a local equation of $D$. 
By assumption the image of $d$ under $f^*: \sO_{X,x}\to \sO_{Y,y}$ is not contained in an associated prime 
of $\sO_{Y,y}$, hence is regular, see \cite[6.]{Mat89}.)
\end{no}

The following definition is reminiscent of the modulus condition introduced in \cite{BE03} and
in its more general form in \cite{BS14}.

\begin{defn}\label{defn:refined-ineq}
Let $X$ be a $k$-scheme and  $D$, $E$ effective Cartier divisors on $X$.
Then we write
\[D\prec E :\Longleftrightarrow \nu^*D\le \nu^* E,\]
where $\nu: \tilde{X}\to X$ is the normalization and $\nu^*D$, $\nu^*E$ is the pullback of Cartier divisors.
(Recall that $\nu$ is the composition $\sqcup_i \tilde{X}_i\xr{\sqcup\nu_i} \sqcup_i X_i\to X_{\rm red}\xr{\iota} X$,
where $\iota$ is the closed immersion of the reduced scheme underlying $X$,
 the $X_i$ are the irreducible components (with  reduced scheme structure) of $X_{\rm red}$ and
$\nu_i:\tilde{X}_i\to X_i$ is the normalization of $X_i$ in its function field.)
\end{defn}

Versions of the following proposition can be found  e.g. in \cite[2.]{KP}.
\begin{prop}\label{prop:refined-ineq}
Let $X$ be a $k$-scheme and  $D$, $E$ effective Cartier divisors on $X$.
\begin{enumerate}
\item\label{prop:refined-ineq1} $D\le E\Longrightarrow D\prec E$.
\item\label{prop:refined-ineq2} If $F$ is another effective Cartier divisor on $X$, then
         \[D\prec E \text{ and } E\prec F \Longrightarrow D\prec F.\]
\item\label{prop:refined-ineq3} Let $n$ be a positive integer. Then 
       \[n D\prec nE\Longleftrightarrow D\prec E.\]
\item\label{prop:refined-ineq4}
           Let $f:Y\to X$ be a morphism of $k$-schemes and assume that $f({\rm Ass}(\sO_X))\subset X\setminus(|E|\cup |D|)$.
          Then the pullbacks $f^*D$ and $f^*E$ exist as effective Cartier divisors (see \ref{no:Cartier-pb}) and 
            \[D\prec E\Longrightarrow f^*D\prec f^*E.\]
          Furthermore if $X$ is irreducible and $f:Y\to X$ is proper and  surjective 
           then this '$\Longleftarrow$' implication also  holds.
\item\label{prop:refined-ineq5}
        Let $Z$ be a normal $k$-scheme and $g:X\to Z$ be a finite and surjective $k$-morphism.
          Then the pushforwards $g_*D$ and $g_*E$ exist as effective Cartier Divisors and 
          we have
         \[D\prec E \Longrightarrow g_*D\prec g_*E.\] 
\end{enumerate}
\end{prop}

\begin{rmk}
Warning: $D\prec E$ and $E\prec D$ does not imply $D=E$. 
            If $X$ is reduced, the kernel of $\nu^*:{\rm CDiv}(X)\to {\rm CDiv}(\tilde{X})$ is given by
         $H^0(X, \nu_*\sO_{\tilde{X}}^\times/\sO_{X}^\times)$, see \cite[Prop (21.8.5)]{EGAIV4}.
                 \end{rmk}

\begin{proof}
Let $X_{\rm red}=\cup_i X_i$ be the decomposition into irreducible components (with reduced scheme structure).
Then by definition of the normalization we have 
$D\prec E$ iff $D_{|X_i}\prec E_{|X_i}$, for all $i$. Hence we can assume that the schemes $X$, $Y$ and $Z$
in the statements are integral $k$-schemes. The statements \eqref{prop:refined-ineq1} and \eqref{prop:refined-ineq2}
are immediate and \eqref{prop:refined-ineq3} follows from the fact that on a normal scheme
the natural map from the group of Cartier divisors to the group of Weil divisors is injective.
We prove '$\Longrightarrow$' of \eqref{prop:refined-ineq4}. The question is local on $Y$.
We may therefore assume that $f:Y\to X$ is induced by a homomorphism between integral $k$-algebras
$\varphi: A\to B$ and $D=\Div (d)$, $E=\Div(e)$, with $d,e\in A\setminus\{0\}$. Then
$D\prec E$ means that the element $e/d\in k(X)^\times$ is integral over $A$, i.e. there exists
 an $n\ge 1$ and $a_i\in A$ with
\begin{align*}
(e/d)^n+a_{n-1} (e/d)^{n-1}+\ldots+a_0=0 \quad \text{in } k(X)\\
\Longleftrightarrow e^n+a_{n-1}d e^{n-1}+\ldots+ a_0 d^n=0 \quad \text{in } A.
\end{align*}
By assumption $\varphi(e),\varphi(d)\in B\setminus\{0\}$.
Applying $\varphi$ to the equation above yields that  $\varphi(e)/\varphi(d)\in k(Y)^\times$ is integral over $B$.
Hence $f^*D\prec f^* E$. 
For the other implication notice that  $f:Y\to X$  also induces a proper, and surjective 
map between the normalizations of $X$ and $Y$. Therefore we can assume that $X$ and $Y$ are normal and integral.
Furthermore the question is local on $X$. Hence we can assume $X=\Spec A$ is affine and
$D=\Div(d)$, $E=\Div(e)$ as above. 
Since $Y$ is normal,  $f^*D\prec f^* E$ is equivalent to $f^*\sO_{X}(D-E)\subset \sO_Y$.
Applying $H^0(Y,-)$ we obtain $e/d\in H^0(Y,\sO_Y)$. Since $f$ is proper, $H^0(Y,\sO_Y)$
is a finite $A$-module. It follows that $e/d$ is integral over $A$. The normality of $A$ gives $e/d\in A$,
i.e. $D\le E$.

Finally for \eqref{prop:refined-ineq5} we first notice that the Cartier divisors
$g_*D$ and $g_*E$ exist by \cite[Def (21.5.5)]{EGAIV4}. If $U\subset Z$ is open and
on $g^{-1}(U)$ we have $D=\Div(d)$   then 
$g_*D=\Div(\Nm(d))$, where $\Nm_{k(X)/k(Z)}:k(X) \to k(Z)$ is the norm map.
It follows from this formula that we have 
$(g\circ\nu)_*\nu^*D=g_*D$, where $\nu:\tilde{X}\to X$ is the normalization.
Since the pushforward $(g\circ\nu)_*$ preserves the ordering on ${\rm CDiv}(X)$ the statement follows.
\end{proof}

\begin{defn}\label{defn:rel-fin-corr}
Let $X$ and $Y$ be $k$-schemes and $D$ and $E$ effective Cartier divisors
on $X$ and $Y$, respectively. Assume that $U=X\setminus|D|$ and $V=Y\setminus|E|$ are smooth
and equidimensional over $k$.
Denote by $\Cor(U, V)$ the group of finite correspondences
from $U$ to $V$, i.e. the free abelian group generated on prime correspondences from $U$ to $V$, i.e 
integral closed subschemes $S\subset U\times V$ which are finite and surjective over a connected component of $U$
(see e.g. \cite[Def 1.1]{MVW}).
The group of finite correspondences from $(X,D)$ to $(Y,E)$ is by definition the subgroup 
\[\Cor((X,D), (Y,E))\subset \Cor(U,V)\]
generated by those prime correspondences $S\subset U\times V$ which satisfy
\eq{defn:rel-fin-corr1}{E_{\ol{S}}\prec D_{\ol{S}},}
where $\ol{S}\subset X\times Y$ is the closure of $S$ and $D_{\ol{S}}$ and $E_{\ol{S}}$ denote
the pullback of $D$ and $E$ along the projections $\ol{S}\inj X\times Y\to X$ and 
$\ol{S}\inj X\times Y\to Y$, respectively, see \ref{no:Cartier-pb}.
\end{defn}

\begin{rmk}
There is a choice in  \eqref{defn:rel-fin-corr1}, one could also ask for the other inequality $\succ$.
Our choice is made in such a way that in case $X$ is {\em smooth} and connected
a finite correspondence $\alpha\in \Cor((X,D), (Y,E))$
induces maps
\[\alpha^*:H^i(Y, \sO_Y(E))\to H^i(X,\sO_X(D))\]
and
\[\alpha_*: H^i(X,\pi_X^!(k)\otimes_{\sO_X}\sO_X(-D))\to H^i(Y,\pi^!_Y(k)\otimes_{\sO_Y}\sO_Y(-E)),\]
where $\pi_X: X\to \Spec k$ and $\pi_Y: Y\to \Spec k$ are the structure maps.
(Notice that since $X$ is smooth we have $\pi_X^!k=\Omega^{\dim X}_X[\dim X]$.)
Indeed if $\alpha=S\subset U\times V$ is a prime correspondence these maps are defined as follows:
Let $\ol{S}\subset X\times Y$ be the closure of $S$ and $\nu:\tilde{S}\to \ol{S}$ its normalization.
Since the  natural map induced by projection $p_X:\tilde{S}\to X$ is surjective and generically finite we
have  a natural map $c_{p_X}:\sO_{\tilde{S}}\to p_X^!\sO_X$ at our disposal, see e.g. \cite[Prop 2.6]{CR15};
by adjointness we obtain a map $p_{X*}:Rp_{X*}\sO_{\tilde{S}}\to \sO_X$.
Then we define  $\alpha^*$ as the composition
\begin{align*}
H^i(Y,\sO_Y(E)) & \xr{p_Y^*} H^i(\tilde{S},\sO_{\tilde{S}}(E_{\tilde{S}}))\\
                       &\xr{\eqref{defn:rel-fin-corr1}}
         H^i(\tilde{S},\sO_{\tilde{S}}(D_{\tilde{S}}))
       \cong H^i(X, (Rp_{X*}\sO_{\tilde{S}})\otimes_{\sO_X}\sO_X(D))\\
          & \xr{p_{X*}} H^i(X, \sO_X(D)).
\end{align*}
Similarly, we define $\alpha_*$ as  the composition
\begin{align*}
H^i(X,\pi_X^!(k)\otimes_{\sO_X}\sO_X(-D))&\xr{\id\to Rp_{X*}Lp_X^*} 
          H^i(\tilde{S}, p_{X}^*(\pi_X^!k)\otimes_{\sO_{\tilde{S}}}\sO_{\tilde{S}}(-D_{\tilde{S}}))\\
&\xr{c_{p_X}} H^i(\tilde{S}, p^!_X(\sO_X)\otimes_{\sO_{\tilde{S}}}p_X^*(\pi_X^!k)
                                         \otimes_{\sO_{\tilde{S}}}\sO_{\tilde{S}}(-D_{\tilde{S}}))\\
 & \cong   H^i(\tilde{S}, \pi_{\tilde{S}}^!(k) \otimes_{\sO_{\tilde{S}}}\sO_{\tilde{S}}(-D_{\tilde{S}}))\\
& \xr{\eqref{defn:rel-fin-corr1}} 
    H^i(\tilde{S}, \pi_{\tilde{S}}^! (k) \otimes_{\sO_{\tilde{S}}}\sO_{\tilde{S}}(-E_{\tilde{S}}))\\
& \cong H^i(Y, Rp_{Y*}p_Y^!(\pi_Y^!k)\otimes_{\sO_Y} \sO_Y(-E))\\
& \xr{Rp_{Y*}p_Y^!\to \id}H^i(Y, \pi_Y^!(k)\otimes_{\sO_Y} \sO_Y(-E)).
\end{align*}
Here for the isomorphism in the third row we use 
$\pi_{\tilde{S}}^!\cong p_X^!\pi_X^!\cong p_X^!(\sO_X)\otimes_{\sO_X}^L Lp_X^*$.
If $\pi_X$ and $\pi_Y$ are projective, then $\alpha_*$ and $\alpha^*$ are dual to each other.
If we chose in \eqref{defn:rel-fin-corr1} this '$\succ$' relation, then we have to switch the signs 
in front of $D$ and $E$ for the maps above.
\end{rmk}

\begin{defn}\label{defn:rel-corr-support}
Let $B$ be a $k$-scheme, $X, Y$  $B$-schemes and 
$D,E$ effective Cartier divisors on $X,Y$, respectively.
Assume that $U=X\setminus|D|$, $V=Y\setminus|E|$ are smooth and equidimensional over $k$.
Then we denote by 
\[\Cor_B((X,D),(Y,E))\]
the subgroup of $\Cor((X,D), (Y,E))$ which consists of all cycles whose support is contained in the closed subscheme
$X\times_B Y\subset X\times Y$. 
\end{defn}

\begin{prop}\label{prop:comp-rel-corr}
Let $B$ be a $k$-scheme, $X, Y,Z$ $B$-schemes and 
$D,E,F$ effective Cartier divisors on $X,Y,Z$, respectively.
Assume that $U=X\setminus|D|$, $V=Y\setminus|E|$ and $W=Z\setminus |F|$ are smooth and equidimensional over $k$
and that $Y$ is {\em proper} over $B$.
Then the usual composition of finite correspondences
$\Cor(U, V)\times \Cor(V,W)\to \Cor(U, W)$, $(\alpha,\beta)\mapsto \beta\circ\alpha$
(see \cite[1.]{MVW})
restricts to a morphism
\eq{prop:comp-rel-corr1}{\Cor_B((X,D),(Y,E))\times \Cor_B((Y,E), (Z,F))\xr{\circ} \Cor_B((X,D), (Z,F)).}
\end{prop}
\begin{proof}
Let $S\subset U\times_B V$ and $T\subset V\times_B W$ be prime correspondence  from
$(X,D)$ to $(Y,E)$ and  from $(Y,E)$ to $(Z,F)$, respectively. 
Let $R\subset S\times_V T$ be an irreducible component 
and denote by $p:S\times_V T\subset U\times_B V\times_B W \to U\times_B W$ the natural  map induced by projection. 
(It is finite.) Then the prime correspondences appearing in $T\circ S$ are $\{p(R)\}_R$, where $R$ runs through all
irreducible components of $S\times_V T$. Hence we have to show that $p(R)$ lies in $\Cor_B((X,D), (Z,F))$.
Denote by $\ol{S}\subset X\times_B Y\subset X\times Y$ and $\ol{T}\subset Y\times_B Z\subset Y\times Z$ 
the closure of  $S$ and $T$, respectively,
and by $\ol{R}\subset \ol{S}\times_Y\ol{T}$ the closure of $R$. 
By definition we have $E_{\ol{S}}\prec D_{\ol{S}}$ and $F_{\ol{T}}\prec E_{\ol{T}}$.
By the first part of \eqref{prop:refined-ineq4} of Proposition \ref{prop:refined-ineq} we have
\[F_{\ol{R}}\prec E_{\ol{R}}\prec D_{\ol{R}}.\]
Denote by $\ol{p}: \ol{S}\times_Y\ol{T}\subset X\times_B Y \times_B Z \to X\times_B Z$ the natural map induced by 
projection. Since $Y$ is proper over $B$ so is $\ol{p}$. 
Hence $\ol{p}_{|\ol{R}}: \ol{R}\to \ol{p(R)}$ is proper, surjective and generically finite.
By the second part of \eqref{prop:refined-ineq4} of Proposition \ref{prop:refined-ineq} we get
\[F_{\ol{p(R)}}\prec D_{\ol{p(R)}},\]
i.e. $p(R)\in \Cor_B((X,D), (Z,F))$.
\end{proof}

\begin{cor}\label{cor:comp-rel-corr}
The composition defined in \eqref{prop:comp-rel-corr1} is associative in the obvious sense.
Furthermore, if $X$ is a $B$-scheme and $D$ is an effective Cartier divisor on $X$ 
such that $X\setminus|D|$ is smooth and equidimensional over $k$, 
then the graph of the diagonal $\Delta_X\subset X\times_B X$
naturally defines an element in $\Cor_B((X,D),(X,D))$. If $X$ is proper over $B$, then
\[\Cor_B((X,D),(Y,E))\xr{\Delta_X\circ } \Cor_B((X,D),(Y,E))\]
is the identity, where $Y$ is a $B$-scheme and $E$ is an effective Cartier divisor $E$ on $Y$
such that $Y\setminus|E|$ is smooth and equidimensional over $k$. 
If $Y$ is proper over $B$, then similarly $\circ \Delta_Y$ is the identity.
\end{cor}
\begin{proof}
This follows immediately from the corresponding properties of the composition for the usual finite
correspondences between smooth $k$-schemes.
\end{proof}

\begin{rmk}\label{rmk:alt-rel-corr}
Let $f:X\to Y$ be an alteration between integral $B$-schemes and assume that
there is an effective Cartier divisor $E$ on $Y$ such that $U:=X\setminus|f^*E|$  and 
$V:=Y\setminus E$ are smooth and $f$ restricts to a finite map $f_{|U}: U\to V$.
Denote by $\Gamma\subset U\times V$ the graph of $f_{|U}$ and by
$\Gamma^t\subset V\times U$ its transpose. 
Then we have 
\[\Gamma \in\Cor_B((X,D), (Y, E)),\quad \text{for all } D\succ f^*E \]
and 
\[^t \Gamma\in \Cor_B((Y,E'),(X,f^*E)),\quad \text{for all }  E'\succ E.\]
If we take $E'= E$ and $D=f^*E$ we get
\[^t\Gamma\circ \Gamma= d\cdot \Delta_X\quad \text{in }\Cor_B((X, f^*E), (X,f^*E)),\]
where $d$ is the generic degree of $f$. 
If $f$ is birational, i.e. $d=1$ we furthermore get
\[\Gamma\circ {}^t\Gamma =\Delta_Y\quad \text{in }\Cor_B((Y, E),(Y,E)).\]
This all follows immediately from the definitions and the corresponding properties for
the usual finite correspondences. 
\end{rmk}

\begin{defn}\label{defn:cat-pairs}
Let $B$ be a $k$-scheme.
We define  $\Sm_B^*\Cor$ to be the category with objects the pairs 
$(X\to B,D)$, where $X\to B$ is a proper $k$-morphism, $D$ is an effective Cartier divisor on $X$ and
$X\setminus|D|$ is smooth and equidimensional over $k$;  the morphisms are by definition
\[\Hom_{\Sm_B^*\Cor}((X\to B,D),(Y\to B,E)):=\Cor_B((X,D),(Y,E)).\]
It follows from Proposition \ref{prop:comp-rel-corr} and Corollary \ref{cor:comp-rel-corr} that this defines
 a category.  We set $\Sm^*\Cor:=\Sm_{\Spec k}^*\Cor$.

For $(X\to B,D), (Y\to B,E)\in\Sm_B^*\Cor$ with $X\setminus|D|$ and $Y\setminus|E|$ 
{\em smooth and equidimensional over $B$} we set
\[(X,D)\times_B (Y,E)= (X\times_B Y, p_X^*D+p_Y^*E),\]
where $p_X: X\times_B Y\to X$ and $p_Y:X\times_B Y\to Y$ are the projections.
It is clear that with this definition the full subcategory of $\Sm_B^*\Cor$ consisting of those pairs
$(X\to B,D)$ with $X\setminus D$ smooth and equidimensional over $B$ has the structure of a symmetric monoidal category 
with the unit element given by $(B, 0)$.  In particular, $\Sm^*\Cor$ has the structure of a symmetric monoidal category.
 \end{defn}

\begin{no}\label{ecubes}
Recall from \cite[1.1]{Le09} that {\bf Cube} is the category with objects given by the sets 
$\ul{n}:=\{0,1\}^n$, $n=0,1,2,\ldots$, and the morphisms are generated by:
\begin{enumerate}
\item  Inclusions $\eta_{n,i,\e}:\ul{n}\inj \ul{n+1}$, 
            $(\e_1,\ldots,\e_n)\mapsto (\e_1,\ldots, \e_{i-1},\e,\e_i,\ldots,\e_n)$, $n\ge 0$, $i\in [0,n+1]$, 
              $\e\in \{0,1\}$.
\item Projections $p_{n,i}:\ul{n}\to \ul{n-1}$, $(\e_1,\ldots, \e_n)\mapsto (\e_1,\ldots, \e_{i-1}, \e_{i+1},\ldots, \e_n)$,
                        $n\ge 1$, $i\in [1,n]$
\item Permutation of factors $\ul{n}\to\ul{n}$, $(\e_1,\ldots,\e_n)\mapsto (\e_{\sigma(1)},\ldots, \e_{\sigma(n)})$,
             for $\sigma$ a permutation of $\{1,2,\ldots,n\}$, $n\ge 1$.
\item Involutions $\tau_{n,i}: \ul{n}\to\ul{n}$ switching $0$ and $1$ in the $i$-th spot, $n\ge 1$, $i\in [1,n]$.
\end{enumerate}
The product of sets induces on {\bf Cube} the structure of a symmetric monoidal category.
The category {\bf ECube} is by definition (see \cite[1.5]{Le09}) the smallest symmetric monoidal subcategory
of the category of sets which has the same objects as ${\bf Cube}$ and
 contains all morphisms in ${\bf Cube}$ and
\[\mu:\ul{2}\to\ul{1}, \quad (\e_1,\e_2)\mapsto \begin{cases} 1, & \text{if } (\e_1,\e_2)=(1,1),\\
                                                                                            0, &\text{else.}\end{cases}\]

We set $\square^1:=\P^1\setminus\{1\}:=\P^1_k\setminus\{1\}$ and 
$\square^n:=(\square^1)^{\times n}= (\P^1\setminus\{1\})^n$. 
In the following we fix coordinates $(\P^1\setminus\{\infty\})^n=\Spec k[y_1,\ldots, y_n]$.
In $\Sm^*\Cor$ we define
\[\ol{\square}^1:= (\P^1, \{1\}), \quad \ol{\square}^n:=(\ol{\square}^1)^{\times n}.\]
Denote by $\square(\eta_{n,i,\e}): \square^n\inj \square^{n+1}$ the inclusion
given by $y_i=\frac{\e}{\e-1}\in \{0,\infty\}$, for $\e\in \{0,1\}$. Denote
by $\square(p_{n,i}): \square^n\to \square^{n-1}$ the projection omitting the $i$-th factor and
by $\square(\sigma):\square^n\to \square^n$ the permutation of the factors given by 
$\square(\sigma)^*(y_i)=y_{\sigma^{-1}(i)}$.
Finally denote by $\square(\tau_{n,i}):\square^n\to \square^n$ the morphism which is the identity on all factors
except the $i$-th factor, where it is induced by the unique isomorphism of $\P^1$ which fixes
 $1$ and switches $0$ and $\infty$. It is straightforward to check that the graphs of these maps
define elements
\[\ol{\square}(\eta_{n,i,\e})\in \Cor(\ol{\square}^n, \ol{\square}^{n+1}),\quad
\ol{\square}(p_{n,i})\in \Cor(\ol{\square}^n, \ol{\square}^{n-1}),\]
\[\ol{\square}(\sigma)\in \Cor(\ol{\square}^n, \ol{\square}^n), \quad 
      \ol{\square}(\tau_{n,i})\in \Cor(\ol{\square}^n, \ol{\square}^n)\]
and that we obtain in this way a strict monoidal functor, in particular a  co-cubical object,
\[\ol{\square}:{\bf Cube}\to \Sm^*\Cor,\quad \ul{n}\mapsto \ol{\square}^n.\]
Finally using the $k$-isomorphism $\P^1\setminus\{1\}\xr{\simeq} \A^1$,
which sends $0$ to $0$ and $\infty$ to $1$ we define $\square(\mu)$ as the composition
\[\square(\mu):\square^2\cong \A^2\xr{\rm multiplication} \A^1\cong \square^1.\]
Let $\ol{\square}(\mu)\subset (\P^1)^2\times \P^1$ be the closure of the graph of $\square(\mu)$.
If we write $(\P^1\setminus\{\infty\})^2\times (\P^1\setminus\{\infty\})=\Spec k[y_1,y_2,z]$, then
\[\ol{\square}(\mu)\cap (\P^1\setminus\{\infty\})^2\times (\P^1\setminus\{\infty\})=
   \Spec k[y_1,y_2,z]/((z-1)y_1y_2-(y_1-1)(y_2-1)z).\]
Hence if we denote by $p_i: \ol{\square}(\mu)\to \P^1$, $i=1,2,3$ the three projections we obtain the
following inequality of Cartier divisors on the smooth scheme $\ol{\square}(\mu)$
\mlnl{p_1^*(\{1\}) +p_2^*(\{1\}) =\\
 (\{1\}\times \P^1\times \{1\})+ (\{1\}\times\{0\}\times\P^1) + (\P^1\times\{1\}\times \{1\}) + 
 (\{0\}\times\{1\}\times \P^1)\\
\ge  (\{1\}\times \P^1\times \{1\})+(\P^1\times\{1\}\times \{1\})= p_3^*(\{1\}).}
Thus 
\[\ol{\square}(\mu)\in \Cor(\ol{\square}^2, \ol{\square}^1).\]
It follows that the functor $\ol{\square}$ extends to a strict monoidal functor
\eq{ecubes1}{\ol{\square}:{\bf ECube}\to \Sm^*\Cor.}
\end{no}

\begin{no}\label{ecubesS}
Let $S$ be a smooth equidimensional $k$-scheme. We set
\[\ol{\square}^n_S:=\ol{\square}^n\times S:=\Big((\P^1)^n\times S, 
\Big(\sum_{i=1}^n (\P^1)^{i-1} \times \{1\} \times (\P^1)^{n-i }\Big)\times S\Big).\]
This is an object in $\Sm^*_S\Cor$, where the structure map $(\P^1)^n\times S\to S$ is given by the projection. 
If we take the cartesian product of the correspondences  
$\ol{\square}(\eta_{n,i,\e})$, $\ol{\square}(p_{n,i})$, $\ol{\square}(\sigma)$, $\ol{\square}(\tau_{n,i})$,
$\ol{\square}(\mu)$ in \ref{ecubes} with the diagonal $\Delta_S\subset S\times S$ 
we obtain correspondences
\[\ol{\square}_S(\eta_{n,i,\e})\in \Cor_S(\ol{\square}^n_S, \ol{\square}^{n+1}_S),\quad
\ol{\square}_S(p_{n,i})\in \Cor_S(\ol{\square}^n_S, \ol{\square}^{n-1}_S),\]
\[\ol{\square}_S(\sigma)\in \Cor_S(\ol{\square}^n_S, \ol{\square}^n_S), \quad 
      \ol{\square}_S(\tau_{n,i})\in \Cor_S(\ol{\square}^n_S, \ol{\square}^n_S),
\quad \ol{\square}_S(\mu)\in \Cor_S(\ol{\square}_S^2, \ol{\square}_S^1).\]
This clearly induces a functor
\[\ol{\square}_S: {\bf ECube }\to \Sm^*_S\Cor.\]
\end{no}

\section{Relative Suslin homology with modulus}\label{sec:rel-Suslin}

The following definition is a modulus version of relative Suslin homology defined in \cite[{\S 3}]{SuVo00}.
\begin{defn}\label{defn:rel-Suslin-homology}
Let $S$ be a smooth equidimensional $k$-scheme and $(X/S, D)\in \Sm^*_S\Cor$.
For $n\ge 0$ set
\[\tilde{C}_n(X/S,D):= \Cor_S(\ol{\square}^n_S, (X,D)).\]
By \ref{ecubesS} we obtain an extended cubical object in the sense of \cite[1.5]{Le09}
\[{\bf ECube}^{\rm opp}\to (\text{\bf abelian groups}), \quad \ul{n}\mapsto \tilde{C}_n(X/S,D).\]
We set 
\[C_n(X/S,D)^{\rm degn}:=\sum_{i=1}^n \im( \ol{\square}_S(p_{n,i})\circ: C_{n-1}(X/S,D)\to C_n(X/S,D)),\]
\[ C_n(X/S,D):= \tilde{C}_n(X/S,D)/C_n(X/S,D)^{\rm degn}\]
and
\[\partial_{n,i,\e}:= \ol{\square}_S(\eta_{n-1,i,\e})\circ \,:C_{n}(X/S,D)\to C_{n-1}(X/S,D).\]
We denote by $(C_\bullet(X/S,D), d)$ the complex with differential given by
\[d_n=\sum_{i=1}^{n} (-1)^i (\partial_{n,i,\infty}-\partial_{n,i,0}): C_n(X/S,D)\to C_{n-1}(X/S,D).\]
Finally for $i\in \Z$ we denote by
\[H_i^S(X/S,D):=H_i(C_\bullet(X/S,D),d),\]
the {\em $i$-th Suslin homology} of $(X/S,D)$.
If $S=\Spec k$ we simply write
\[H^S_i(X,D):=H_i^S(X/\Spec k,D).\]
\end{defn}

\begin{rmk}\label{rmk:Cn-expl}
Notice that by definition an integral correspondence
$Z\in C_n(X/S,D)$ can be identified with an integral closed subscheme 
$Z\subset (\P^1\setminus\{1\})^n\times (X\setminus|D|)$ such that
the map $Z\to (\P^1\setminus\{1\})^n\times S$ induced by $X\to S$ is finite and surjective and 
if $\ol{Z}\subset (\P^1)^n\times X$ denotes the closure of $Z$ then 
\eq{rmk:Cn-expl1}{D_{\ol{Z}}\prec \Big(\sum_{i=1}^n (\P^1)^{i-1}\times\{1\}\times (\P^1)^{n-i}\Big)_{\ol{Z}}.}
Also notice that because of this last condition $Z$ is actually closed in $(\P^1\setminus\{1\})^n\times X$.

Furthermore by definition (or convention) the group 
$C_0(X/S,D)$ is equal to the closed integral subschemes $Z\subset X$ which are finite and
surjective over a connected component of $S$ and are contained in $X\setminus|D|$.
\end{rmk}

\begin{lem}[Levine]\label{lem:NSH}
Let $S$ be a smooth $k$-scheme and $(X/S, D)\in \Sm^*_S\Cor$.
For $n\ge 0$ set
\[NC_n(X/S,D):= \bigcap_{i=2}^{n}\Ker(\partial_{n,i,0})\cap\bigcap_{i=1}^n\Ker (\partial_{n,i,\infty}).\]
Then the differential $d_n: C_n(X/S,D)\to C_{n-1}(X/S,D)$ restricts to
\[\partial_{n,1,0} : NC_n(X/S,D)\to NC_{n-1}(X/S,D)\]
and the natural inclusion of complexes
\[(NC_\bullet(X/S,D),\partial_{\bullet,1,0})\inj (C_\bullet(X/S,D),d)\]
is a homotopy equivalence. In particular
\[H_i^S(X/S,D)=H_i(NC_\bullet(X/S,D)).\]
\end{lem}
\begin{proof}
This follows from the fact that $\ul{n}\mapsto \tilde{C}_n(X/S,D)$ is an extended cubical object
and \cite[Lem 1.6]{Le09}.
\end{proof}

\begin{no}\label{propertiesRSH}
We give some further properties of the groups $H_i^S(X/S,D)$.
\begin{enumerate}
\item $\Sm_S^*\Cor\to D^{-}(\text{\bf abelian groups})$, $(X/S,D)\to C_\bullet(X/S,D)$ is a covariant functor.
        In particular, $\Sm_S^*\Cor\to (\text{\bf abelian groups})$, $(X/S,D)\to H^S_i(X/S,D)$ is a covariant functor for all
           $i$.
\item\label{propertiesRSH1} Let ${\rm CH}^r(X|D,n)$ be the higher Chow groups with modulus from \cite{BS14} 
       which are the $n$-th homology groups of the cycle complex $z^r(X|D,\bullet)$ defined in \cite[Def 2.5]{BS14}.
       If $d=\dim X$, then there is a natural map 
           $C_\bullet(X,D)\to z^d(X|D,\bullet)$ inducing maps 
                 \eq{propertiesRSH11}{H_i^S(X,D)\to {\rm CH}^d(X|D,i),}
                  cf. \cite[Conj 1.2]{SuVo00}.
\item\label{propertiesRSH2} 
    For $i=0$, the map $H_0^S(X,D)\xr{\eqref{propertiesRSH11}} {\rm CH}^d(X|D)={\rm CH}_0(X|D)$ is an isomorphism.
         Indeed, this follows directly from the definition and the fact that a 1-dimensional closed subscheme
            in $(X\setminus|D|)\times \square^1$, which intersects the faces of $\square^1$ properly, has either 
             trivial boundary or is finite and surjective over $\square^1$. (Here we use that $X$ is proper over $k$.)
\item For $(X/S,D)\in \Sm_S^*\Cor$   and $D'$ an effective Cartier divisor on $X$ with $|D'|=|D|$ and $D'\ge D$
        the identity on $X$ induces a map $(X/S, D')\to (X/S,D)$ in $\Sm_S^*\Cor$.  Now let $U$ be an $S$-scheme which
 is smooth and equidimensional over $k$. Using Nagata compactification we find a pair $(X/S, D)\in \Sm_S^*\Cor$
     with an $S$-isomorphism $U\cong X\setminus|D|$. By the comment above we obtain a projective
       system $(H^S_i(X/S, m\cdot D))_{m\in \N}$. We set 
               \[H_{c,i}^S(U/S):=\varprojlim_m H^S_i(X/S,m\cdot D).\]
        Notice that this definition is independent of the choice $(X/S,D)$. Indeed as usual 
         this follows from the two observations: 1. Any two choices of $(X/S,D)$ and $(X'/S, D')$ 
         can be dominated by a third object  $(X''/S, D'')\in \Sm_S^*\Cor$ with $X''\setminus|D''|\cong U$. 
          2. If $(X'/S, D')\to (X/S, D)$ is a morphism in $\Sm^*_S\Cor$ inducing an $S$-isomorphism on 
              $X'\setminus|D'|\xr{\simeq} X\setminus|D|$, then it also induces an isomorphism
                      \[\varprojlim_m H^S_i(X'/S,m\cdot D')\xr{\simeq} \varprojlim_m H^S_i(X/S,m\cdot D).\]
                 The latter fact follows directly from Remark \ref{rmk:alt-rel-corr} (the case of a birational map).
\item Assume $k$ is a finite field of characteristic $\neq 2$ and $U$ a smooth and integral $k$-scheme.
           Then using \eqref{propertiesRSH2} above 
         and \cite[Thm 3.3]{BS14} the Theorem of Kerz and Saito \cite[Thm III]{KeS} can be reformulated as an 
          isomorphism
                   \[H_{c,0}^S(U)\xr{\simeq } W^{\rm ab}(U),\]
             where $W^{\rm ab}(U)\subset \pi_1^{\rm ab}(U)$ is the abelianized Weil group of $U$.
\item\label{propertiesRSH6} 
Let $(X/S,D)\in \Sm_S^*\Cor$ and put $U=X \setminus |D|$.
The relative Suslin homology $H_i^S(U/S)$ 
is defined in \cite[{\S 3}]{SuVo96} as the homology group
of a simplicial complex $C_\bullet^S(U/S)$.
As usual it can be rewritten as a homology group
of a cubical complex.
Namely, by forgetting modulus
in the construction of ${C}_{\bullet}(X/S, D)$,
one gets a complex $C_{\bullet}(U/S)$ 
which is quasi-isomorphic to $C_\bullet^S(U/S)$.
Then $C_{\bullet}(X/S, D)$ is a subcomplex of $C_{\bullet}(U/S)$,
hence there is a canonical homomorphism
$H_i^S(X/S, D) \to H_i^S(U/S)$.
It is not clear if this is bijective when $D$ is reduced
(compare \cite[Remark 3.5]{KSY}).
\end{enumerate}
\end{no}

\section{The main theorem}

We state the main result of this note 
in Theorem \ref{thm:RSHcurves} below.

\begin{no}\label{not:OXD}
We recall some notations from \cite{SuVo96} and \cite{KSY}.
For $X$ an integral $k$-scheme and $D$ an effective Cartier divisor on $X$ we denote
\[\sO_{X|D}^\times:=\Ker(\sO_X^\times\to \sO_D^\times).\]
We set ${\rm Pic}(X, D):=H^1(X, \sO^\times_{X|D})$. Recall that this group can be identified with the group
of isomorphism classes of pairs $(\sL, \alpha)$, where $\sL$ is a line bundle on $X$ and $\alpha$ is an isomorphism
$\sL_{|D}\xr{\simeq}\sO_D$. We denote by $\widetilde{{\rm Pic}}(X, D)\subset {\rm Pic}(X, D)$ 
the subgroup of isomorphism classes of pairs $(\sL,\alpha)$, for which there exists an open subset $V\subset X$ with
$|D|\subset V$ and such that $\alpha$ is the restriction of an isomorphism $\sL_{|V}\cong\sO_{V}$.
If $D$ has an open affine neighborhood, then $\widetilde{{\rm Pic}}(X, D)={\rm Pic}(X, D)$, see \cite[2.]{SuVo96}.
Further, we set
\[G(X,D):=\bigcap_{x\in|D|} \Ker(\sO_{X,x}^\times\to \sO_{D,x}^\times).\]
We denote by $X^{(1)}$ the set points $x\in X$ whose closure $\ol{\{x\}}\subset X$ has codimension one. 
Notice that if $X$ is normal and we write $D=\sum_i n_i D_i$, where the $D_i$ are the irreducible components
of $D$, and if we denote by $v_x: k(X)^\times\to \Z$ the normalized discrete valuation defined by a
1-codimensional point $x\in X^{(1)}$, then
\mlnl{G(X,D)=\\
\{f\in k(X)^\times\,|\,  v_x(f)= 0, x\in X^{(1)}\text{ with } \ol{\{x\}}\cap D\neq\emptyset, 
   \text{ and }v_x(f-1)\ge n_i, \text{ if } \ol{\{x\}}=D_i\}.}
Indeed, since $X$ is normal the condition $v_x(f)= 0$, for all $x\in X^{(1)}$ with $\ol{\{x\}}\cap D\neq\emptyset$, implies that 
$f$ is regular on a neighborhood containing $D$
(e.g. the complement of the pole divisor of $f$). Further since $D$ is a Cartier divisor it has no embedded components; hence if a function is regular in a neighborhood of $D$ and its image in $\sO_D$ is equal to 1 in all generic points of $D$, 
then it is equal to 1 on all of $D$. Finally we denote by 
\[{\rm Div}(X,D)\]
the group of Cartier divisors on $X$ with support in $X\setminus|D|$.
By \cite[Lem 3.2]{KSY} there is an exact sequence
\eq{not:OXD1}{0\to H^0(X, \sO_{X|D}^\times)\to G(X,D)\xr{\Div} {\rm Div}(X,D)\to \widetilde{\rm Pic}(X,D)\to 0.}
\end{no}

\begin{no}\label{assumptions}
For the rest of  this note, we fix a smooth connected $k$-scheme $S$ 
and $(\sC/S, D)\in \Sm_S^*\Cor$ such that:
\begin{enumerate}
\item $\sC$ is integral and normal.
\item $\sC\to S$ is proper and the generic relative dimension is equal to 1, i.e. ${\rm trdeg}(k(\sC)/k(S))=1$.
\item $U=\sC\setminus D$ is smooth over $k$ and $U\to S$ is affine.
\item There is an open affine subscheme of $\sC$ which contains all the generic points of $D$.  
\end{enumerate}

\end{no}

\begin{lem}\label{lem:Div-C}
Under the above assumptions, we have  $C_0(\sC/S,D)= {\rm Div}(\sC,D)$.
\end{lem}
\begin{proof}
First recall that $C_0(\sC/S,D)$ is the free abelian group generated by integral closed subschemes $Z\subset U$
which are finite and surjective over $S$. Since $U$ is smooth and $U\to S$ has generic relative dimension 1, 
$Z$ has to be a Cartier divisor on $U$ which is actually closed in $\sC$. Therefore 
$C_0(\sC/S,D)$ is a subgroup of ${\rm Div}(\sC,D)$.
On the other hand if $Z\subset \sC$ is a prime Cartier divisor whose support is contained
in $U$, then by the assumptions in \ref{assumptions} it has to be affine and proper over $S$ and hence finite.
Therefore its image is a reduced closed subscheme $Z_0\subset S$. 
If $Z_0$ is strictly contained in $S$, then  the dimension formula yields ${\rm trdeg}(k(Z)/k(Z_0))\ge 1$.
But since $Z$ is finite over $S$ this is impossible. Therefore $Z$ is surjective over $S$ and hence an element in $C_0(\sC/S,D)$.
\end{proof}

The following theorem is a modulus version of \cite[Thm. 3.1]{SuVo96}
(see also \cite{Li}).

\begin{thm}\label{thm:RSHcurves}
Let $(\sC\to S,D)$ be as in \ref{assumptions}. Then 
\[C_\bullet(\sC/S,D)\cong [G(\sC,D)\xr{\Div} {\rm Div}(\sC,D)]\quad \text{in } D^{-}(\text{{\bf abelian groups}}),\]
where the complex on the right sits in homological degree 1 and 0.
In particular, by \eqref{not:OXD1} we have
\[H_i^S(\sC/S,D)=\begin{cases} \widetilde{{\rm Pic}}(\sC, D), &\text{if } i=0,\\
                                             H^0(\sC, \sO^\times_{\sC|D}), & \text{if }i=1,\\
                                               0, &\text{if } i\ge 2.
                                             \end{cases}\]
\end{thm}

\begin{rmk}\label{rmk:simplecase}
If $D$ is contained in an affine open neighborhood, then $H^0(X,\sO_{X|D}^\times)$ vanishes 
(see \cite[Proof of Thm 3.1]{SuVo96}) and $\widetilde{\rm Pic}(\sC,D)={\rm Pic}(\sC,D)$.
This shows Theorem \ref{thm-intro}.
\end{rmk}

The proof of Theorem \ref{thm:RSHcurves} will occupy the rest of this note.
We start with some preliminary observations.

\section{An auxiliary construction}
\begin{no}
We fix the following notation for this section:
Let $A$ be a semilocal Dedekind domain 
(thus a PID) 
with field of fractions $K={\rm Frac}(A)$.
We denote  by $\fm_1,\ldots, \fm_r$ the maximal ideals of $A$ and by
$v_1,\ldots, v_r: K^\times\to \Z$ the corresponding normalized discrete valuations.
 We fix an element $s\in \fm_1\cdots\fm_r$ and a set of normalized discrete valuations $\sV$
on $K$, which does not contain $v_1,\ldots, v_r$. 

Given a discrete valuation $v$ on $K$ there exists a {\em canonical extension} of $v$ to the field
of rational functions $K(t_1,\ldots, t_n)$,
which we denote again by $v$, such that for all
$0\neq f=\sum a_{i_1,\ldots, i_n} t^{i_1}_1\cdots t_n^{i_n}\in K[t_1,\ldots, t_n]$ we have
\eq{aux1}{v(f)= \min\{v(a_{i_1,\ldots, i_n}),\text{ all }i_1,\ldots, i_n\},}
see e.g. \cite[VI, \S 10.1, Prop 2]{BourbakiCA}.
\end{no}

\begin{defn}\label{defn:Qn}
For $n\ge 1$ we define $Q_n(A,\sV, s):=Q_n$ to be the set of all polynomials
\eq{defn:Qn0}{f=\sum_{i_1=0}^{N_1}\cdots\sum_{i_n=0}^{N_n} a_{i_1,\ldots, i_n}\cdot t_1^{i_1}\cdots t_n^{i_n}\in
     A[t_1,\ldots, t_n]}
such that
\begin{enumerate}
\item\label{defn:Qn1} $a_{i_1,\ldots, i_n}\in s^{\max_{j=1}^n(N_j-i_j)}\cdot A$, all $i_1,\ldots,i_n$.
\item\label{defn:Qn2} $a_{N_1,\ldots, N_n}\in A^\times$.
\item\label{defn:Qn2.5} $v(f/a_{N_1,\ldots, N_n})=0$, for all $v\in \sV$.
\item\label{defn:Qn3} For all $j\in [1,n]$ we have     $f(t_1,\ldots, t_{j-1},0,t_{j+1},\ldots, t_n)\neq 0$. 
\end{enumerate}
Notice that $A^\times\subset Q_n$. 
\end{defn}

\begin{defn}\label{defn:LC}
Let $f\in A[t_1,\ldots, t_n]$ be a polynomial as in \eqref{defn:Qn0}. Then we say that
$f$ has a leading coefficient if $a_{N_1,\ldots, N_n}\neq 0$, in which case we call $a_{N_1,\ldots, N_n}$ the leading
coefficient of $f$.
\end{defn}

\begin{defn}\label{defn:rho-f}
Let $f\in A[t_1,\ldots, t_n]$ be a polynomial and assume it satisfies the conditions \eqref{defn:Qn1},\eqref{defn:Qn2}
and \eqref{defn:Qn2.5} of Definition \ref{defn:Qn}. Then we denote by $\rho(f)$ the unique polynomial
which satisfies
\begin{enumerate}
\item $\rho(f)\in Q_n$.
\item $f=t_1^{m_1}\cdots t_n^{m_n}\cdot \rho(f)$, for non-negative integers $m_1,\ldots, m_n$.
\end{enumerate} 
\end{defn}

\begin{lem}\label{lem:Qn-via-Qn-1}
Assume $n\ge 2$. For $j\in [1,n]$ denote by $A^{(j)}$ the localization of $A[t_j]$ with respect to
$\cup_{i=1}^r (\fm_i\cdot A[t_j])$. If $v$ is a valuation on $K$ denote by $v^{(j)}$ its canonical extension to $K(t_j)$ and
set $\sV^{(j)}:=\{v^{(j)}\,|\, v\in \sV\}$. Denote by
\[\iota_j: A[t_1,\ldots, t_n]\inj A^{(j)}[t_1,\ldots, \widehat{t_j},\ldots, t_n]\]
the natural inclusion, where the hat $\widehat{(\phantom{-})}$ means omission. In particular,
$A^{(j)}$ is a semilocal Dedekind domain with maximal ideals $\fm_i^{(j)}=\fm_i A^{(j)}$, 
and $s\in \fm_1^{(j)}\cdots \fm_r^{(j)}$.
We denote by $\eqref{defn:Qn1}^{(j)}$, $\eqref{defn:Qn2}^{(j)}$ the properties
\eqref{defn:Qn1}, \eqref{defn:Qn2} from Definition \ref{defn:Qn} with
$A,\sV$ replaced by $A^{(j)},\sV^{(j)}$.
Then, for $f \in A[t_1,\ldots, t_n]$ we have
\[f \text{ satisfies } \eqref{defn:Qn1} \text{ and } \eqref{defn:Qn2}\Longleftrightarrow 
  \iota_j(f) \text{ satisfies } \eqref{defn:Qn1}^{(j)} \text{ and }\eqref{defn:Qn2}^{(j)}, \quad\text{for all }j\in[1,n].\]
\end{lem}
\begin{proof}
For $j\in [1,n]$ write
\[f=\sum_{i_1=0}^{N_1}\cdots\widehat{\sum_{i_j=0}^{N_j}}\cdots \sum_{i_n=0}^{N_n} 
a_{i_1,\ldots, \widehat{i_j},\ldots, i_n}^{(j)} t_{1}^{i_1}\cdots \widehat{t_j^{i_j}}\cdots t_n^{i_n},\]
where
\[a_{i_1,\ldots, \widehat{i_j},\ldots, i_n}^{(j)}= \sum_{i_j=0}^{N_j} a_{i_1,\ldots, i_n} t_j^{i_j}\in 
    A[t_j]\subset A^{(j)}.\]
 We write 
$\eqref{defn:Qn1}^{(j)}(\iota_j(f))$ for '$\iota_j(f)$ satisfies $\eqref{defn:Qn1}^{(j)}$' etc.
Then,
\begin{align*}
&\eqref{defn:Qn1}^{(j)}(\iota_j(f))
\\
\Longleftrightarrow &  v_e^{(j)}(a_{i_1,\ldots, \widehat{i_j},\ldots, i_n}^{(j)})\ge 
       ( \max_{c\in [1,n]\atop c\neq j}(N_c-i_c ))\cdot v^{(j)}_e(s),\quad  
     \text{all }e\in[1,r], \,i_1,\ldots,\widehat{i_j},\ldots i_n
\\
\Longleftrightarrow & \min_{i_j=0}^{N_j}(v_e(a_{i_1,\ldots, i_n})) \ge 
         ( \max_{c\in [1,n]\atop c\neq j}(N_c-i_c ))\cdot v_e(s),\quad 
                  \text{all }e\in[1,r],\, i_1,\ldots,\widehat{i_j},\ldots i_n.
\end{align*}
Hence
\[\eqref{defn:Qn1}^{(j)}(\iota_j(f)), \quad \text{for all }j\in[1,n] \Longleftrightarrow \eqref{defn:Qn1}(f).\]
Furthermore,
\begin{align*}
\eqref{defn:Qn2}^{(j)}(\iota_j(f)) &\Longleftrightarrow 
                v_e^{(j)}(a^{(j)}_{N_1,\ldots,\widehat{N_j},\ldots,N_n})=0,\quad \text{all }e\in [1,r]\\
   &\Longleftrightarrow \min_{i_j=0}^{N_j}(v_e(a_{N_1,\ldots,i_j,\ldots, N_n}))=0,\quad \text{all }e\in [1,r]\\
 &\stackrel{\eqref{defn:Qn1}}{\Longleftrightarrow} v_e(a_{N_1,\ldots, N_n})=0,\quad \text{all }e\in [1,r]\\
  &\Longleftrightarrow \eqref{defn:Qn2}(f).
\end{align*}
Hence the lemma.
\end{proof}

\begin{lem}\label{lem:eltsQn}
Let $f\in A[t_1,\ldots, t_n]$ be a polynomial.
Then $f$ satisfies properties \eqref{defn:Qn1} and \eqref{defn:Qn2} of Definition \ref{defn:Qn} if and only if
$f$ satisfies the following two properties:
\begin{enumerate}
\item\label{lem:eltsQn1} There exist  $N\in \Z_{\ge 0}$ and  $g_1,\ldots,g_n,h\in A[t_1,\ldots, t_n]$ such that
         \[f\cdot h= (t_1\cdots t_n)^N+s\cdot (t_1\cdots t_n)^{N-1}\cdot g_1+\cdots + s^N \cdot g_N.\]
\item\label{lem:eltsQn2} $f(t_1,\ldots, t_{j-1},\infty,t_{j+1},\ldots, t_n)\not\in \fm_iA[t_1,\ldots, \widehat{t_j}, \ldots, t_n]$, 
          for all $j\in [1,n]$ and     $i\in [1,r]$.
         (Here $f(t_1,\ldots, t_{j-1},\infty,t_{j+1},\ldots, t_n)$ denotes the leading term of $f$ viewed as a polynomial in $t_j$.)
\end{enumerate}
\end{lem}
\begin{proof}
Assume $f$ satisfies the conditions \eqref{lem:eltsQn1} and \eqref{lem:eltsQn2} from the statement.
We have to show that $f$ has the properties \eqref{defn:Qn1} and \eqref{defn:Qn2} of Definition \ref{defn:Qn}.
We prove this by induction over $n$.
First assume $n=1$ and write $t=t_1$. 
In this case clearly \eqref{lem:eltsQn2} of the statement implies
\eqref{defn:Qn2} of Definition \ref{defn:Qn}.
Set $g(t):= t^N+s\cdot t^{N-1}\cdot g_1(t)+\cdots + s^N \cdot g_N(t)\in A[t]$.
Let $b$ be a root of $f$ in some finite extension $L$ of $K$. By $\eqref{lem:eltsQn1}$ we have 
$g(b)=0$ and hence the element $b$ lies in the integral closure $B$ of $A$ in $L$. 
By the Krull-Akizuki theorem $B$ is still a semi-local Dedekind domain. Using that $B$ is factorial and the special shape of $g$
it is direct to check that $s|b$. Taking $L$ large enough we can assume that all roots of $f$ lie in $B$. Hence
$f=u\prod_i(t-sb_i)$, with $u\in A^\times$ and $b_i\in B$. This implies $f$ satisfies \eqref{defn:Qn1} of Definition \ref{defn:Qn}. 
Now if $n\ge 2$ we observe that for all $j\in [1,n]$ the polynomials
$\iota_j(f)\in A^{(j)}[t_1,\ldots, \widehat{t_j},\ldots, t_n]$ satisfy the analogs of \eqref{lem:eltsQn1},\eqref{lem:eltsQn2}
of the statement, where we use the notation from Lemma \ref{lem:Qn-via-Qn-1}.
(Notice that $t_j^N$ is a unit in $A^{(j)}$.)
Hence by induction $\iota_j(f)$ satisfies $\eqref{defn:Qn1}^{(j)}$, $\eqref{defn:Qn2}^{(j)}$ of Definition \ref{defn:Qn} 
for all $j\in [1,n]$. We conclude with Lemma \ref{lem:Qn-via-Qn-1}.

Conversely, if $f$ has the properties \eqref{defn:Qn1} and \eqref{defn:Qn2} of Definition \ref{defn:Qn}.
Then it clearly satisfies \eqref{lem:eltsQn2} of the statement.
We show that it also satisfies \eqref{lem:eltsQn1}. To this end let $f$ be as in \eqref{defn:Qn0}.
Set $M:=\max\{N_1,\ldots, N_n\}$. Then by condition \eqref{defn:Qn1} of Definition \ref{defn:Qn} 
for all tuples $(i_1,\ldots, i_n)\neq (N_1,\ldots, N_n)$ there exist an $i\in [0,M-1]$ such that $\max_{j=0}^n(N_j-i_j)=M-i$
and $a_{i_1,\ldots, i_n}= s^{M-i} b_{i_1,\ldots, i_n}\in A$
with $b_{i_1,\ldots, i_n}\in A$.
We get
\eq{lem:eltsQn3}{f\cdot (a_{N_1,\ldots, N_n}^{-1} t_1^{M-N_1}\cdots t_n^{M-N_n})=
 (t_1\cdots t_n)^M+ \sum_{i=1}^{M} s^{M-i} (t_1\cdots t_n)^i\cdot g_{M-i},}
where
 \[g_{M-i}=\sum_{(i_1,\ldots, i_n)\atop \max_{j=1}^n(N_j-i_j)=M-i} b_{i_1,\ldots, i_n} t_1^{M-i-N_1+i_1}\cdots t_n^{M-i-N_n+i_n}
\in A[t_1,\ldots, t_n].\]
Thus $f$ satisfies \eqref{lem:eltsQn1}.
\end{proof}

\begin{lem}\label{lem:prodQn}
Let $f,g\in A[t_1,\ldots, t_n]$ be polynomials. Then
\[f\cdot g\in Q_n\Longleftrightarrow f, \,g \in Q_n.\]
\end{lem}
\begin{proof}
This '$\Longleftarrow$' direction is easy. We prove the other direction.
Clearly $f$ and $g$ satisfy condition \eqref{defn:Qn3} of Definition \ref{defn:Qn}.
By Lemma \ref{lem:eltsQn} $f\cdot g\in Q_n$ implies that $f\cdot g$
satisfies the conditions \eqref{lem:eltsQn1},\eqref{lem:eltsQn2} of Lemma \ref{lem:eltsQn}
and then also each of  $f$ and $g$ satisfies these conditions. Hence 
$f$ and $g$ satisfy \eqref{defn:Qn1}, \eqref{defn:Qn2} of Definition \ref{defn:Qn}.
Furthermore if $a,b$ are the leading coefficients of $f,g$, then $ab$ is the leading coefficient of $f g$ and we have
by assumption $0=v(fg/ab)=v(f/a)+v(g/b)$, for all $v\in \sV$. Since $f/a$, $g/b$ have leading coefficient equal to 1
we have $v(f/a)\le 0$ and $v(g/b)\le 0$, by the formula in \eqref{aux1}. 
Hence $v(f/a)=v(g/b)=0$. Thus  $f,g \in Q_n$.
\end{proof}

\begin{cor}\label{cor:Qn-free-monoid}
$Q_n/A^\times$ is the free abelian monoid generated on polynomials $f\in Q_n$ which are irreducible and are normalized 
such that the leading coefficient is 1.
\end{cor}
\begin{proof}
This follows from Lemma \ref{lem:prodQn} and the fact that $A[t_1,\ldots, t_n]$ is factorial.
\end{proof}

\section{Proof of the main theorem}
\begin{no}\label{corr-poly}
Let $(\sC/S, D)$ be as in \ref{assumptions}.
Let $D_{\rm red}=\cup_i D_i$ be the decomposition into irreducible components.
We denote by $A$ the semilocalization of $\sO_\sC$ at the generic points of $D$.
(By assumption \ref{assumptions} (4) there is an open affine subscheme $\Spec B\subset \sC$ containing all 
the generic points of $D$, then $A$ is the localization of $B$ at the union of the prime ideals
corresponding to these points.)
In particular, $A$ is  a semilocal Dedekind domain. Let $s\in A$ be an equation for $D$, i.e.
\[D_{|\Spec A}=\Spec A/(s).\]
 Set $U:= \sC\setminus |D|$ and
\[\sV:=\{v_x\,|\, x\in U^{(1)} \text{with }\ol{\{x\}}\cap |D|\neq \emptyset\},\]
where we denote by $v_x:k(\sC)^\times\to \Z$ the normalized discrete valuation corresponding to $x$.
We have two maps
\eq{corr-poly0}{U\times\square^n\xr{j}\sC\times(\P^1)^n\xl{\alpha} \Spec A\times \A^n,}
where $j$ is the natural open immersion and $\alpha$ is the product of the natural map
$\Spec A\to \sC$ and the composition 
\eq{corr-poly0.5}{\A^n=\Spec k[t_1,\ldots, t_n]\to \Spec k[y_1,\ldots, y_n]=(\P^1\setminus\{\infty\})^n\inj (\P^1)^n,}
where the first map is induced by $y_i\mapsto 1-t_i$.

Let $Z\subset U\times \square^n$ be an integral closed subscheme which is finite and surjective over
$S\times\square^n$ and denote by $\ol{Z}\subset \sC\times(\P^1)^n$ its closure.
In particular $\codim(\ol{Z}, \sC\times(\P^1)^n)=1$.
Denote by $Z_0\subset\sC$ the image of $\ol{Z}$. Since $Z_0$ surjects onto $S$ the dimension formula yields
\[1= {\rm trdeg}(k(Z_0)/k(S))+ \codim(Z_0,\sC).\]
Therefore we arrive in one of the following two cases:
\begin{enumerate}
\item\label{corr-poly1} $\codim(Z_0, \sC)=1$ and hence $\ol{Z}=Z_0\times(\P^1)^n$.
\item\label{corr-poly2} $Z_0=\sC$.
\end{enumerate}
In the second case the scheme-theoretic inverse image $\alpha^{-1}(\ol{Z})$ is an integral closed subscheme
of codimension one in $\Spec A \times\A^n$. Since $A[t_1,\ldots, t_n]$ is factorial there is an  irreducible
polynomial $p(Z)\in A[t_1,\ldots, t_n]$ such that
\eq{corr-poly3}{\alpha^{-1}(\ol{Z})=\Div(p(Z)).}
Notice that $p(Z)$ is only well-defined up to multiplication with an element from $A^\times$.
\end{no}

\begin{lem}\label{lem:descr-corr}
Let $Z\subset U\times\square^n$ be as above. 
If $Z\in \tilde{C}_n(\sC/S,D)$, then one of the following two conditions is satisfied (with the notations from above):
\begin{enumerate}
\item $\codim(Z_0,\sC)=1$ and $Z_0\in C_0(\sC/S,D)$ (i.e. $Z_0$ is finite and surjective over $S$ and contained in $U$).
\item $Z_0=\sC$ and $p(Z)\in Q_n(A,\sV,s)$, with $A$, $\sV$ and $s$ as chosen above.
\end{enumerate}
\end{lem}
\begin{proof}
In case $\codim(Z_0,\sC)=1$ we have $\ol{Z}=Z_0\times(\P^1)^n$.
Hence \eqref{rmk:Cn-expl1} implies $Z_0\cap D=\emptyset$, i.e. $Z_0\subset U$. 
Now assume that $Z_0=\sC$. We have to check that 
\[f:=p(Z)\]
 satisfies \eqref{defn:Qn1}-\eqref{defn:Qn3} of Definition \ref{defn:Qn}. 
In fact \eqref{defn:Qn3} is clear since we start with $Z\subset U\times\square^n$.
For  \eqref{lem:eltsQn2} of Lemma \ref{lem:eltsQn} we have to show that $\ol{Z}$ does not contain
a closed subset of the form $D_i\times (\P^1)^{j-1}\times\{\infty\}\times (\P^1)^{n-j}$, for some $i,j$. 
But this is excluded by the modulus condition \eqref{rmk:Cn-expl1}.
Pulling back the modulus condition along $\alpha$ from \eqref{corr-poly0} we obtain
\eq{lem:descr-corr1}{D_{\alpha^{-1}(\ol{Z})}\prec V(t_1\cdots t_n)_{\alpha^{-1}(\ol{Z})}.}
By definition of $f$ in \eqref{corr-poly3} this means that the element 
\[\frac{t_1\cdots t_n}{s}\in {\rm Frac}\left(\frac{A[t_1,\ldots, t_n]}{(f)} \right)\]
is integral over $A[t_1, \ldots, t_n]/(f)$.
This directly translates into \eqref{lem:eltsQn1} of Lemma \ref{lem:eltsQn}.
Hence $f$ satisfies the conditions \eqref{defn:Qn1} and \eqref{defn:Qn2} of Definition \eqref{defn:Qn}.
Finally, if $x\in U^{(1)}$ with $\ol{\{x\}}\cap D\neq \emptyset$ we want to show that 
$v_x(f/a)=0$, where $a$ is the leading coefficient of $f$. We may assume $a=1$
and hence we have to show $v_x(f)=0$. Take $y\in \ol{\{x\}}\cap|D|$ and 
set $B:=\sO_{\sC,y}$ and $V:=\Spec B$. Denote by $\fc_{y}(f)\subset k(\sC)$ the $B$-submodule 
given by 
\[\fc_y(f)=\{c\in k(\sC)\,|\, c\cdot f\in B[t_1,\ldots, t_n]\}.\]
Then $\ol{Z}\times_{\sC\times (\P^1)^n} (V\times \A^n)$ is the closed subscheme of $\Spec B[t_1,\ldots, t_n]$
given by the ideal $\fc_y(f)\cdot f\cdot B[t_1,\ldots, t_n]$.
Since the leading coefficient of $f$ is equal to 1, we have $\fc_y(f)\subset B$.
The modulus condition \eqref{rmk:Cn-expl1} tells us 
in particular that $\ol{Z}\cap (|D|\times \square^n)=\emptyset$.
Hence  $\ol{Z}\times_{\sC\times(\P^1)^n} (\{y\}\times (\A^1\setminus\{0\})^n)=\emptyset$. If $\fm \subset B$ denotes the maximal ideal
this means that 
\[\frac{\fc_y(f)\cdot f\cdot B[t_{1}^{\pm 1},\ldots, t_n^{\pm 1}]}{\fm \cdot B[t_{1}^{\pm 1},\ldots, t_n^{\pm 1}]}
=\frac{B[t_{1}^{\pm 1},\ldots, t_n^{\pm 1}]}{\fm \cdot B[t_{1}^{\pm 1},\ldots, t_n^{\pm 1}]},\]
i.e. there exists an element $b\in\fc_y(f)\subset B$ 
and $M_1, \dots, M_n \in \Z$
such that 
\eq{lem:descr-corr2}{b\cdot f\in   (t_1^{M_1}\cdots t_n^{M_n})\cdot B^\times +\fm B[t_{1}^{\pm 1},\ldots, t_n^{\pm 1}]. }
Now write $f$ as in \eqref{defn:Qn0} (with $a_{N_1,\ldots, N_n}=1$). We claim that $M_i=N_i$ for all $i$.
Indeed, \eqref{lem:descr-corr2} yields that $b\cdot a_{M_1,\ldots, M_n}\in B^\times$. If $D_j$ is an irreducible component
of $D$ passing through $y$ and with generic point $\eta_j$, then we have a 
natural map $B\to \sO_{\sC,\eta_j}$  and hence $b\cdot a_{M_1,\ldots, M_n}\in \sO_{\sC,\eta_j}^\times$.
In particular $a_{M_1,\ldots, M_n}\in \sO_{\sC, \eta_j}^\times$ which by \eqref{defn:Qn1} of Definition \ref{defn:Qn}
is only possible if $M_i=N_i$. Since $a_{N_1,\ldots, N_n}=1$ \eqref{lem:descr-corr2} yields that $b\in B^\times$ and hence $\fc_y(f)= B$.
If $\pi\in \sO_{\sC, x}$ is a local parameter then we obtain
\[\pi^{-v_x(f)}\cdot\sO_{\sC,x}= \{c\in k(\sC)\,|\, c\cdot f\in \sO_{\sC,x}[t_1,\ldots, t_n]\}=\fc_y(f)\cdot\sO_{\sC,x}
=\sO_{\sC,x},\]
i.e. $v_x(f)=0$. This proves the lemma.
\end{proof}

\begin{prop}\label{prop:descr-corr}
Denote by $\tilde{C}_n(\sC/S,D)^{{\rm eff}}$ the free submonoid  of $\tilde{C}_n(\sC/S, D)$ generated by prime 
correspondences. Then with the notation from \ref{corr-poly} above there is an isomorphism of monoids for all $n\ge 1$
\eq{prop:descr-corr0}{\tilde{C}_n(\sC/S,D)^{{\rm eff}}\xr{\simeq} C_0(\sC/S,D)^{{\rm eff}}\oplus (Q_n(A,\sV,s)/A^\times), }
which is given by 
\[Z\subset U\times\square^n \text{(integral)}\mapsto
          \begin{cases}  Z_0, & \text{if }\ol{Z}=Z_0\times(\P^1)^n,\\
                              p(Z) \text{ mod }A^\times, &\text{else.} \end{cases}\]
\end{prop}
\begin{proof}
The map is well-defined by Lemma \ref{lem:descr-corr}. It is clearly injective and surjects onto the generators of the first factor.
Let $f\in Q_n(A,\sV,s)$ be a non-constant irreducible polynomial. Then with the notation from \ref{corr-poly} it defines a point 
$(f)\in \Spec A\times \A^n$ and we denote by $\ol{Z}$ the closure of $\alpha((f))$ in $\sC\times (\P^1)^n$ and 
set $Z:=j^{-1}(\ol{Z})$. We claim 
\eq{prop:descr-corr1}{Z\in \tilde{C}_n(\sC/S,D).}
Then clearly $p(Z)=f$ mod $A^\times$ and the surjectivity follows from Corollary \ref{cor:Qn-free-monoid}.
We prove \eqref{prop:descr-corr1}. Since $f$ satisfies \eqref{defn:Qn3} of Definition \ref{defn:Qn}, we see that
$Z$ is not empty and is henceforth an integral closed subscheme of codimension 1 inside $U\times\square^n$.
Since $f$ is non-constant $\ol{Z}$ surjects onto $\sC$. Further the property \eqref{defn:Qn2} of Definition \ref{defn:Qn} 
implies that $\ol{Z}$ does not contain closed subschemes of the form
$D_i\times (\P^1)^{j-1}\times\{\infty\}\times (\P^1)^{n-j}$, for any $i,j$.
Hence there is no irreducible component of $\ol{Z}\cap (D\times(\P^1)^n)$
 which maps to $\infty\in \P^1$ under any of the $n$ projections.
Thus it suffices to check the modulus condition \eqref{rmk:Cn-expl1} after pulling $\ol{Z}$ back along 
the composition
\[\sC\times \A^n\xr{\id_\sC\times\eqref{corr-poly0.5}}\sC\times (\P^1\setminus\{\infty\})^n\inj \sC\times(\P^1)^n.\]
To this end write $f$ as in \eqref{defn:Qn0}; we may assume $a_{N_1,\ldots, N_n}=1$.
Property \eqref{defn:Qn2.5} of Definition \ref{defn:Qn} then implies 
that there exists an open subset $V\subset \sC$ with $D\subset V$ such that 
$a_{i_1,\ldots, i_n}\in \sO_{\sC}(V)$. (Here we use that $\sC$ is normal.)
Set $M:=\max\{N_1,\ldots, N_n\}$ and denote by $\eta_j\in D$ the generic points.
Then by \eqref{defn:Qn1} of Definition \ref{defn:Qn} the $a_{i_1,\ldots, i_n}$ map to zero in 
$\sO_{(M-i)\cdot D,\eta_j}$ for all $j$ and all tuples $(i_1,\ldots, i_n)$ with $\max_{j=1}^n(N_j-i_j)= M-i$, for some
$i\in [0,M-1]$. Since $D$ has no embedded components, this implies that
$a_{i_1,\ldots, i_n}\in H^0(V, \sO_\sC(-(M-i)\cdot D))$, for these tuples $(i_1,\ldots, i_n)$.
Therefore, if $V'\subset V$ is an open affine subset on which $D$ is given by an equation $\sigma$ one 
gets a similar equation as in  \eqref{lem:eltsQn3} and concludes 
 that $t_1\cdots t_n/\sigma$ is integral over $\sO(V')[t_1,\cdots, t_n]/(f)$.
This directly translates into the modulus condition $\eqref{rmk:Cn-expl1}$ over $V'$. 
Since $V'\subset V$ was arbitrary the modulus condition holds in general. 
In particular $Z$ is closed in $\sC\times \square^n$ and hence is proper over $S\times\square^n$.
Since $U\to S$ is affine so is $Z\to S\times\square^n$ and therefore $Z$ is finite over $S\times \square^n$.
Then $Z$ is also surjective onto $S\times \square^n$ for dimension reasons and hence $Z\in \tilde{C}_n(\sC/S,D)$.
\end{proof}

\begin{no}\label{Z(f)}
Let $f\in Q_{n+1}$ be a non-constant irreducible polynomial and write $f$ as in \eqref{defn:Qn0}.
Then for $j\in [1, n+1]$ and $\e\in \{0,\infty\}$ we define the correspondence
\[Z_{j,\e}(f)\in C_0(\sC/S,D)^{\rm eff}\]
as follows: For a prime correspondence $Z\in C_0(\sC/S,D)$ denote by
$v_Z: k(\sC)^\times\to \Z$ the corresponding normalized discrete valuation;
it extends canonically to valuations on 
$k(\sC\times\square^n)$ and on $k(\sC\times\square^{n+1})$, see \ref{aux1}.
Then 
\[Z_{j,\e}(f):=\sum_{Z\in C_0(\sC/S,D)} (-v_Z(f)+b_{j,\e}(f)_Z) \cdot Z,\]
with
\[b_{j,\e}(f)_Z:= v_{Z}(f(t_1,\ldots, t_{j-1}, \e+1, t_{j},\ldots, t_{n})).\]
(By convention we put $\infty+1 =\infty$. We also use the notation from Lemma \ref{lem:eltsQn}.)
Notice that $b_{j,\e}(f)_Z\ge \min\{v_Z(a_{i_1,\ldots, i_{n+1}})\}$ and hence $-v_Z(f)+b_{j,\e}(f)_Z\ge 0$, for all $Z$. 
Also for $u\in A^\times$ we have $Z_{j,\e}(f)=Z_{j,\e}(uf)$. 
If $f=\prod_i f_i\in Q_{n+1}$ with $f_i$ irreducible then we set
$Z_{j,\e}(f)=\sum_i Z_{j,\e}(f_i).$
By Corollary \ref{cor:Qn-free-monoid} this yields a morphism of monoids 
\[Z_{j,\e}: Q_{n+1}/A^\times\to C_0(\sC/S,D)^{\rm eff}.\]
Denote by $(Q_{n+1}/A^\times)^{\rm gp}$ the group generated by the monoid $Q_{n+1}/A^\times$.
Then we obtain a group homomorphism
\[Z_{j,\e}: (Q_{n+1}/A^\times)^{\rm gp}\to C_0(\sC/S,D), \quad f/g\mapsto Z_{j,\e}(f)-Z_{j,\e}(g).\]
\end{no}

\begin{lem}\label{lem:Qn-boundary}
The following diagram commutes 
for all $n\ge 1$, $j \in [1, n]$ and $\e\in \{0,\infty\}$
\[\xymatrix{\tilde{C}_{n+1}(\sC/S,D)\ar[d]_{\eqref{prop:descr-corr0}}^{\simeq}\ar[rr]^{\partial^j_{\e}}& &
                                \tilde{C}_n(\sC/S,D)\ar[d]_{\simeq}^{\eqref{prop:descr-corr0}}\\
                  C_0(\sC/S,D)\oplus (Q_{n+1}/A^\times)^{\rm gp}\ar[rr] 
                            &         &  C_0(\sC/S,D)\oplus (Q_{n}/A^\times)^{\rm gp},}\]
where the lower map is given by 
\[\left(Z,\frac{f(t_1,\ldots, t_{n+1})}{g(t_1,\ldots, t_{n+1})}\right)\mapsto 
 \left(Z+Z_{j,\e}(f/g), 
  \frac{\rho(f(t_1,\ldots, t_{j-1}, 1+\e, t_{j},\ldots, t_n))}{\rho(g(t_1,\ldots, t_{j-1}, 1+\e, t_{j},\ldots, t_n))}\right),\]
with $Z_{j,\e}(f)$ as defined in \ref{Z(f)} and $\rho(f)$ as in Definition \ref{defn:rho-f}.
In particular the following diagram commutes:
\[\xymatrix{\tilde{C}_1(\sC/S,D)\ar[rrrr]^{\partial_0-\partial_\infty}\ar[d]_{\simeq} & & & & C_0(\sC/S,D)\ar[d]^{\simeq}\\
             C_0(\sC/S,D)\oplus (Q_1/A^\times)^{\rm gp}\ar[rrrr]^-{(Z,f(t))\mapsto Z_{1,0}(f)-Z_{1,\infty}(f) }& & & &
                                  C_0(\sC/S,D).
}\]
\end{lem}
\begin{proof}
Let $Z\in C_{n+1}(\sC/S,D)$ be a prime correspondence. If $Z=Z_0\times\square^{n+1}$ with $Z_0\in C_0(\sC/S,D)$
a prime correspondence, then $\partial^j_\e(Z)=Z_0\times\square^n$. Therefore  
$Z\mapsto \partial^j_\e(Z)$ translates under  \eqref{prop:descr-corr0} into $(Z_0,0)\mapsto (Z_0,0)$.
Now assume $Z$ dominates $\sC$ and set $f:=p(Z)$ with the notation from  \eqref{corr-poly3}
(this is well defined up to multiplication with an element from $A^\times$).
Denote by $Z'\subset U\times (\P^1\setminus\{0\})^{n+1}$ the image of $Z$ under the isomorphism
\eq{lem:Qn-boundary1}{U\times(\P^1\setminus\{1\})^{n+1}\xr{\simeq} U\times (\P^1\setminus\{0\})^{n+1},}
which is induced by the unique isomorphism of $\sO_U$-algebras
 $\sO_U[t_1,\ldots, t_{n+1}]\to \sO_U[y_1,\ldots, y_{n+1}]$ given by $t_i\mapsto 1-y_i$.
We have to compute $\partial^j_{1+\e}(Z')$. We restrict to the case $\e=0$ and $j=n+1$, the other cases are similar.
Set $B=H^0(U,\sO_{\sC})$ and $K={\rm Frac}(B)={\rm Frac}(A)$.
Then by definition of $f$ the ideal of $Z'$ in $U\times (\P^1\setminus\{0,\infty\})^{n+1}$ is given by
\[\fp= \{cf\,|\, c\in K \text{ with } cf\in B[t_1,\ldots, t_{n+1}] \}\cdot 
                                B[t_1,\ldots, t_{n+1},\frac{1}{t_1\cdots t_{n+1}}].\]
Let $x$ be the generic point of a prime correspondence in $C_n(\sC/S,D)$ and 
denote by $x'\in U\times (\P^1\setminus\{0\})^n$ its image under the isomorphism \eqref{lem:Qn-boundary1}.
Denote by $R$ the localization of $B[t_1,\ldots, t_{n+1}]$ with respect to the prime ideal 
defined by $x'\times \{1\}$ and by $v_{x'}: {\rm Frac}(R/(t_{n+1}-1))^\times\to \Z$ the valuation
defined by $x'$. Then the multiplicity with which the correspondence $\ol{\{x'\}}$ occurs
in $\partial_1^{n+1}(Z')$ is equal to $v_{x'}(\fp R/(t_{n+1}-1))$. 
(Notice that $\fp R/(t_{n+1}-1)$ is a principal ideal so that the latter expression makes sense.)
Now if $x'$ dominates $\sC$ we have $K\subset R$ so that
$\fp R= (f)$ and 
\[v_{x'}(\fp R/(t_{n+1}-1))=v_{x'}(f(t_1,\ldots, t_n,1))=v_{x'}(\rho(f(t_1,\ldots, t_n,1))).\] 
Otherwise there is a point $x_0\in U^{(1)}$ which dominates $S$ and such that
$x'$ is the generic point of $x_0\times (\P^1\setminus\{0,\infty\})^n$.
In this case $v_{x'}$ extends the valuation $v_{x_0}$ defined by $x_0$
and if $\pi\in \sO_{U, x_0}$ is a local parameter we have
\[\fp R= (\pi^{-v_{x'}(f)}\cdot f).\]
Hence $v_{x'}(\fp R/(t_{n+1}-1))= -v_{x'}(f)+ b_{n+1,0}(f)_{\ol{\{x'\}}}$.
This proves the lemma.
\end{proof}

\begin{lem}\label{lem:Q1-G}
Let $f\in Q_1$ be non-constant. Then $f(1)/f(\infty)\in G(\sC,D)$ (see \ref{not:OXD} for the notation)
and $Z_{1,0}(f)-Z_{1,\infty}(f)=\Div (f(1)/f(\infty))$.
\end{lem}
\begin{proof}
The first part is immediate from Definition \ref{defn:Qn}. The second part follows from the exact sequence
\eqref{not:OXD1}, Lemma \ref{lem:Div-C} and 
the definition of $Z_{j,\e}(f)$ in \ref{Z(f)}.
\end{proof}

\begin{defn}\label{defn:NQn} For $n\ge 1$ we define $NQ_n(A,\sV,s):=NQ_n$ to be the set of
all elements in ${\rm Frac}(A[t_1,\ldots, t_n])^\times/A^\times$
such that:
\begin{enumerate}
\item\label{defn:NQn1} 
it has a representation $f/g$ with $f,g\in Q_n$
satisfying \eqref{defn:NQn2}-\eqref{defn:NQn4} below.
\item\label{defn:NQn2} 
\[\beta_{j,0}(f/g):=\frac{\rho(f(t_1, \ldots, t_{j-1}, 1, t_{j+1}, \ldots, t_n))}{\rho(g(t_1, \ldots, t_{j-1}, 1, t_{j+1}, \ldots, t_n))}
                  \in A^\times,\quad \text{for all } j\in [2,n].\]
\item\label{defn:NQn3} 
\[\beta_{j,\infty}(f/g):=
  \frac{\rho(f(t_1, \ldots, t_{j-1}, \infty, t_{j+1}, \ldots, t_n))}{\rho(g(t_1, \ldots, t_{j-1}, \infty, t_{j+1}, \ldots, t_n))}
                  \in A^\times,\quad \text{for all } j\in [1,n].\]
\item\label{defn:NQn4} \[\beta_{j,\e}(f/g)/\beta_{j',\e'}(f/g)\in H^0(\sC, \sO^\times_{\sC|D}),\]
where $(j,\e), (j',\e')\in ([1,n]\times\{\infty\})\cup ([2,n]\times\{0\})$
 and we use the notation from \ref{not:OXD}.
\end{enumerate}
Notice that $NQ_n$ is a group.
We obtain a homological complex
\[NQ_\bullet: \ldots\to NQ_n\xr{\delta_n} NQ_{n-1}\xr{\delta_{n-1}} \ldots \to NQ_1\xr{\delta_1} NQ_0:=C_0(\sC/S,D)\to 0,\]
where 
\[\delta_n\left(\frac{f(t_1,\ldots, t_n)}{g(t_1,\ldots, t_n)}\right)=\begin{cases} 
      \frac{\rho(f(1, t_1,\ldots, t_{n-1}))}{\rho(g(1, t_1,\ldots, t_{n-1}))}, &\text{if } n\ge 2\\
                                                       \Div(\frac{f(1)}{g(1)}\frac{g(\infty)}{f(\infty)}),&\text{if }n=1. \end{cases}\]
\end{defn}

\begin{cor}\label{cor:NQn-corr}
The isomorphism from Proposition \ref{prop:descr-corr} induces an isomorphism 
\[NC_\bullet(\sC/S,D)\xr{\simeq} NQ_\bullet,\]
where $NC_\bullet(\sC/S,D)$ is the normalized chain complex from Lemma \ref{lem:NSH}. 
In particular, by the latter lemma $NQ_\bullet$ is homotopy equivalent to $C_\bullet(\sC/S,D)$.
\end{cor}
\begin{proof}
Consider the homological complex $C_0(\sC/S,D)\oplus (Q_\bullet/A^\times)^{\rm gp}$ with differentials as described in
Lemma \ref{lem:Qn-boundary}. It suffices to show that the normalized complex of 
$C_0(\sC/S,D)\oplus (Q_\bullet/A^\times)^{\rm gp}$ (in the sense of Lemma \ref{lem:NSH}) is equal to
$NQ_\bullet$. To this end, we take $(\alpha, f/g)\in C_0(\sC/S,D)\oplus (Q_n/A^\times)^{\rm gp}$.
We can assume that $f$ and $g$ have leading coefficient equal to 1. 
Assume that $\partial^j_\e (\alpha, f/g)=0$ for all 
$(j,\e)\in ([1,n]\times\{\infty\})\cup ([2,n]\times\{0\})$.
By Lemma \ref{lem:Qn-boundary} this is equivalent to 
$\beta_{j,\e}(f/g)\in A^\times$ and
$\alpha= -Z_{j,\e}(f/g)$ 
for all $(j,\e)\in( [1,n]\times\{\infty\})\cup ([2,n]\times\{0\})$.
Finally notice that under the above conditions
\[Z_{j,\e}(f/g)=Z_{j',\e'}(f/g)\]
is by the definition of $Z_{j,\e}$ in \ref{Z(f)} equivalent to
\[\Div(\beta_{j,\e}(f/g))=\Div(\beta_{j',\e'}(f/g)) \quad \text{on }\sC.\]
Hence $\beta_{j,\e}(f/g)/\beta_{j',\e'}(f/g)$ is a global unit of $\sC$. Furthermore since $f$ and $g$ have leading coefficient 
equal to 1 
condition \eqref{defn:Qn1} of Definition \ref{defn:Qn} says that $f$ and $g$ map to 1 in the residue fields
of the generic points of $D$. Since $D$ is a Cartier divisor, it has no embedded components and therefore
$f=g=1$ on all of $D$. This gives $\beta_{j,\e}(f/g)/\beta_{j',\e'}(f/g)\in H^0(\sC,\sO_{\sC|D}^\times)$.

We define $\theta_0: NQ_0 \to C_0(\sC/S,D)$ to be the identity map and for  $n\ge 1$,
\[\theta_n: NQ_n\to C_0(\sC/S,D)\oplus (Q_n/A^\times)^{\rm gp}, \quad f/g\mapsto (-Z_{1,\infty}(f/g), f/g).\]
We claim that the diagram 
\eq{cor:NQn-corr1}{\xymatrix{NQ_n\ar[r]^-{\theta_n}\ar[d]_{\delta_n} & 
                  C_0(\sC/S,D)\oplus (Q_n/A^\times)^{\rm gp}\ar[d]^{\partial^1_0} \\
              NQ_{n-1}\ar[r]_-{\theta_{n-1}} &     C_0(\sC/S,D)\oplus (Q_{n-1}/A^\times)^{\rm gp}                } }
commutes for all $n\ge 1$. It is then clear that $\theta$ induces an isomorphism from $NQ_\bullet$ 
to the normalized complex of $C_0(\sC/S,D)\oplus (Q_\bullet/A^\times)^{\rm gp}$, which proves the corollary.
For $n=1$ the commutativity of $\eqref{cor:NQn-corr1}$ follows directly from the definition of  $\delta_1$,
Lemma \ref{lem:Qn-boundary} and Lemma \ref{lem:Q1-G}. For $n\ge 2$ the commutativity amounts to show
\eq{cor:NQn-corr2}{-Z_{1,\infty}(f/g)+Z_{1,0}(f/g)=-Z_{1,\infty}(f(1,t_1,\ldots, t_{n-1})/g(1,t_1,\ldots, t_{n-1})).}
To this end, let $Z\in C_0(\sC/S,D)$ be a prime correspondence. 
With the notation from \ref{Z(f)} we compute
\[-( -v_Z(f/g)+b_{1,\infty}(f/g)_Z)+ (-v_Z(f/g)+b_{1,0}(f/g)_Z) \]
\[ = - v_{Z}\left(\frac{f(\infty, t_1,\ldots, t_{n-1})}{g(\infty, t_1,\ldots, t_{n-1})}\right)
                                          +  v_{Z}\left(\frac{f(1, t_1,\ldots, t_{n-1})}{g(1, t_1,\ldots, t_{n-1})}\right)\]
\[= -v_{Z}\left(\frac{f(1,\infty, t_1,\ldots, t_{n-2})}{g(1, \infty, t_1,\ldots, t_{n-2})}\right) 
        +v_Z\left(\frac{f(1, t_1,\ldots, t_{n-1})}{g(1, t_1,\ldots, t_{n-1})}\right)\]
\[= -\left(b_{1,\infty}\left(\frac{f(1, t_1,\ldots, t_{n-1})}{g(1, t_1,\ldots, t_{n-1})}\right)_Z
               -v_Z\left(\frac{f(1, t_1,\ldots, t_{n-1})}{g(1, t_1,\ldots, t_{n-1})}\right)\right),\]
where for the second equality we used  
 that $f/g$ has the properties \eqref{defn:NQn3} and \eqref{defn:NQn4} 
of Definition \ref{defn:NQn}. This gives \eqref{cor:NQn-corr2} and finishes the proof.
\end{proof}

\begin{lem}\label{lem:N-prop}
Assume $n\ge 2$ and  take $f,g\in Q_n$ with $f/g\in NQ_n$.
Denote by $a\in A^\times$ the leading coefficient of $f$ and by $b\in A^\times$ the leading coefficient of $g$ 
(see Definition \ref{defn:Qn}). Then for all $(j,\e)\in( [1,n]\times\{\infty\})\cup ([2,n]\times\{0\})$
we have
\eq{lem:N-prop1}{\frac{f(t_1, \ldots, t_{j-1}, 1+\e, t_{j+1}, \ldots, t_n)}{g(t_1, \ldots, t_{j-1}, 1+\e, t_{j+1}, \ldots, t_n)}
           =\frac{a}{b}\cdot t_1^{m_1}\cdots\widehat{t_j^{m_j}}\cdots t_n^{m_n},}
where $m_i=\deg_{t_i}(f)- \deg_{t_i}(g_i)$, $i\in [1,n]$.
If furthermore $\delta_n(f/g)=1$, then \eqref{lem:N-prop1} also holds for $(j,\e)=(1,0)$. 
\end{lem}
\begin{proof}
Write
\[f=\sum_{i_1=0}^{N_1}\cdots\sum_{i_n=0}^{N_n} a_{i_1,\ldots, i_n}\cdot t_1^{i_1}\cdots t_n^{i_n}\]
and
\[g=\sum_{j_1=0}^{M_1}\cdots\sum_{j_n=0}^{M_n} b_{j_1,\ldots, j_n}\cdot t_1^{j_1}\cdots t_n^{j_n}.\]
Then 
\[f(t_1,\ldots, t_{n-1},\infty)=\sum_{i_1=0}^{N_1}\cdots\sum_{i_{n-1}=0}^{N_{n-1}} a_{i_1,\ldots, i_{n-1},N_n}
                             \cdot t_1^{i_1}\cdots t_{n-1}^{i_{n-1}}\]
and 
\[g(t_1,\ldots, t_{n-1},\infty)=\sum_{j_1=0}^{M_1}\cdots\sum_{j_{n-1}=0}^{M_{n-1}} b_{j_1,\ldots, j_{n-1},M_n}
                             \cdot t_1^{i_1}\cdots t_{n-1}^{i_{n-1}}.\]
Since $n\ge 2$ we can only have an equality $\beta_{n,\infty}(f/g)=c_{n,\infty}\in A^\times$ if there exist
non-negative integers $r_i\in[0, N_i]$ and $s_i\in [0, M_i]$ such that 
\eq{lem:N-prop2}{a_{i_1+r_1,\ldots, i_{n-1}+r_{n-1}, N_n}= c_{n,\infty}\cdot b_{i_1+s_1,\ldots, i_{n-1}+ s_{n-1}, M_n},\quad 
  \text{for all }i_1,\ldots,i_{n-1}.}
We set $m_i(n):=r_i-s_i$. Then plugging in  $i_j=N_j-r_j$ we see that 
\[b_{N_1-m_1(n),\ldots, N_n-m_n(n)}\in A^\times.\]
By definition of $Q_n$ this is only possible if
\eq{lem:N-prop3}{\deg_{t_i}(f)=N_i=M_i+m_i(n)=\deg_{t_i}(g)+m_i(n),\quad \text{for }i\in [1,n-1].}
Furthermore we get
\[\beta_{n, \infty}(f/g)=
c_{n,\infty}=a_{N_1,\ldots, N_n}/b_{M_1,\ldots, M_n}=a/b.\]
We can play the same game with $\beta_{j,\infty}(f/g)$ for all $j\in [1,n]$ and get
similar as in \eqref{lem:N-prop3}
\[m_i(j)=N_i-M_i, \quad \text{for all } i\in [1,n]\setminus\{j\},\, j\in [1,n].\]
Hence $m_i(j)$ is constant for $i$ fixed and $j$ running; we set $m_i:=m_i(j)$.
Furthermore similar as above we get
\[\beta_{j,\infty}(f/g)=a/b, \quad \text{for all }j\in [1,n].\]
This proves \eqref{lem:N-prop1} for $(j,\e)\in [1,n]\times\{\infty\}$.
The equality $\beta_{n,0}(f/g)=c_{n,0}\in A^\times$ implies
that there exist non-negative integers $r_i'\in [0,N_i]$ and $s_i'\in[0, M_i]$ such that
\[\sum_{i=0}^{N_n}a_{i_1+r_1',\ldots, i_{n-1}+r_{n-1}', i}= c_{n,0}\cdot 
 \sum_{i=0}^{M_n} b_{i_1+s_1',\ldots, i_{n-1}+ s_{n-1}', i},\quad 
  \text{for all }i_1,\ldots,i_{n-1}.\]
If we plug in $N_j-r'_j$ for the $i_j$'s, the left hand side of the above equality is in $A^\times$
(by \eqref{defn:Qn1} and \eqref{defn:Qn2} of Definition \ref{defn:Qn}).
Hence the right hand side is also a unit in this case; this is only possible if
\[b_{N_1-r'_1+s_1',\ldots, N_{n-1}-r'_{n-1}+ s_{n-1}', M_n}\in A^\times\]
which is only possible if
\[N_i-(r'_i-s'_i)=M_i,\quad \text{for all } i\in [1,n-1].\]
Hence $r'_i-s'_i=m_i$ for $i\in [1,n-1]$. 
As above $\beta_{1,\infty}(f/g)=a/b$ implies , cf. \eqref{lem:N-prop2},
\[a_{N_1,\ldots,N_{n-1}, i}=\begin{cases}  a/b\cdot b_{M_1,\ldots, M_{n-1}, i-m_n}, & \text{if } i\ge m_n,\\
                                                           0, & \text{if } i<m_n. \end{cases}\]
Hence
\mlnl{a/b\cdot \sum_{i=0}^{M_n} b_{M_1,\ldots, M_{n-1}, i}  = \sum_{i=0}^{M_n} a_{N_1,\ldots, N_{n-1}, i+m_n}
\\=\sum_{i=0}^{N_n} a_{N_1,\ldots, N_{n-1}, i}
= c_{n,0}\cdot \sum_{i=0}^{M_n}b_{M_1,\ldots, M_{n-1}, i}.}
Hence $c_{n,0}=a/b$. Similarly $\beta_{j,0}(f/g)\in A^\times$ implies $\beta_{j,0}(f/g)=a/b$ for
all $j\in [2,n]$ and if $\delta_n(f/g)=1$ also for $j=1$. Hence the lemma.
\end{proof}

\begin{proof}[Proof of Theorem \ref{thm:RSHcurves}]
We start by showing that $H_n^S(\sC/S,D)=0$, for $n\ge 2$.  Take $f,g\in Q_n$ with $f/g\in NQ_n$ and assume
$\delta_n(f/g)=1$. We have to show that $f/g$ lies in the image of $\delta_{n+1}: NQ_{n+1}\to NQ_n$.
To this end we may assume that the leading coefficients of $f,g$ are equal to 1. We denote by
$N_i=\deg_{t_i}(f)$ and $M_i=\deg_{t_i}(g)$ their degrees. Set $m_i:=N_i-M_i$.
We define
\[m_i^+:=\begin{cases} m_i,&\text{if } m_i\ge 0,\\ 0,&\text{if }m_i<0\end{cases},\quad
  m_i^-:=\begin{cases} 0,&\text{if } m_i\ge 0,\\ -m_i,&\text{if }m_i<0.\end{cases}\]
and
\[h_1(t_1,\ldots, t_{n+1}):=  t_1\cdot (\prod_{i=2}^{n+1} t_i^{m_{i-1}^-}) \cdot f(t_2,\ldots, t_{n+1}),\]
\[h_2(t_1,\ldots, t_{n+1}):=   (\prod_{i=2}^{n+1} t_i^{m_{i-1}^-})\cdot f(t_2,\ldots, t_{n+1})- 
        (\prod_{i=2}^{n+1} t_i^{m_{i-1}^+})\cdot g(t_2,\ldots, t_{n+1}).\]
Then $h_1$ satisfies properties \eqref{defn:Qn1}-\eqref{defn:Qn2.5} of Definition \ref{defn:Qn}.
If we set $R_i:=\deg_{t_i}(h_1)$, then $R_1=1$ and $R_i=\max\{N_{i-1},M_{i-1}\}$, for $i\in [2,n+1]$.
By definition we have $\deg_{t_i}(h_2)\le R_i$ and the coefficients of $h_2$ before $\prod_{i=1}^{n+1}t_i^{R_i}$ and
$\prod_{i=2}^{n+1} t_i^{R_i}$ are zero. If we have a tuple $(i_2,\ldots, i_{n+1})\in \prod_{i=2}^{n+1}[1, R_i]$
in which $i_j<R_j$ for at least one $j$, and $c_{0, i_2, \ldots, i_{n+1}}$ is the coefficient of $h_2$ in front of
$\prod_{j=2}^{n+1} t_j^{i_j}$, then 
\[c_{0, i_2, \ldots, i_{n+1}}\in s^{\max_j\{R_j-i_j\}}\cdot A.\]
This follows from $f,g\in Q_n$ and 
\[\min\{\max_j\{N_{j-1}-(i_j-m_{j-1}^-)\},\max_j\{M_{j-1}-(i_j-m_{j-1}^+)\}\}=\max_j\{R_j-i_j\}.\]
Furthermore condition \eqref{defn:Qn2.5} of Definition \ref{defn:Qn} for $f,g$ clearly implies
$v(c_{0, i_2, \ldots, i_{n+1}})\ge 0$ for all $v\in \sV$.
Hence $h_1-h_2$ satisfies condition  \eqref{defn:Qn1} - \eqref{defn:Qn2.5} of Definition \ref{defn:Qn}.
Thus 
\[h:=\rho(h_1-h_2) \in Q_{n+1}.\]
We claim
\[q:=f(t_2,\ldots, t_{n+1})/h(t_1,\ldots, t_{n+1})\in NQ_{n+1}.\]
Indeed, since $f/g\in NQ_n$ and $\delta_n(f/g)=1$ it follows from Lemma \ref{lem:N-prop} that
under $t_i\mapsto 1+\e$ the polynomial $h_2$ is mapped to zero, for all $i\in [2,n+1]$ and $\e\in \{0,\infty\}$.
Hence with the notation from Definition \ref{defn:NQn} we have 
\[\beta_{j,\e}(q)=1,\quad \text{for all } j\in [2,n+1],\,\e\in\{0,\infty\}.\]
One checks directly that the same holds for $j=1$ and $\e=\infty$. 
Hence $q\in NQ_{n+1}$. Finally we observe
\[\delta_{n+1}(q)=\frac{f(t_1,\ldots, t_{n})}{g(t_1,\ldots, t_n)}.\]
This finishes the prove of $H^S_i(\sC/S,D)=0$ for $i\ge 2$.

We are left to show that $NQ_1/\im(\delta_2) \xr{\delta_1} NQ_0$ is quasi-isomorphic
to $G(\sC,D)\to {\rm Div}(\sC, D)$. 
First we claim that for $f,g\in Q_1$ we have
\eq{pfthm:RSHcurves1}{\frac{f(1)}{f(\infty)}=\frac{g(1)}{g(\infty)}\Longleftrightarrow f/g\in \im(\delta_2).}
Indeed, to see this '$\Longrightarrow$' direction we may assume that 
$f(\infty)=g(\infty)=1$ and that $n:=\deg(f)\ge \deg(g):=m$. 
Set
\[h(t_1,t_2):=\frac{f(t_2)}{\rho(t_1f(t_2)-(f(t_2)-t_2^{n-m}g(t_2)))}.\]
As above one checks that $h\in NQ_2$ and $\delta_2(h)=f/g$. 
The other direction follows immediately from Lemma \ref{lem:N-prop}.

We define the following map
\[\varphi: G(\sC,D)\to NQ_1/\im(\delta_2), \quad a\mapsto \rho((1-t_1)-a) \text{ mod }\im(\delta_2).\]
Notice that $\varphi$ is a group homomorphism since $f(t_1)=\rho((1-t_1)-ab)$ and 
$g(t_1)=\rho((1-t_1)-a)\rho((1-t_1)-b)$ satisfy the condition on the left of \eqref{pfthm:RSHcurves1}.
By '$\Longleftarrow$' of \eqref{pfthm:RSHcurves1} $\varphi$ is injective.
By '$\Longrightarrow$' of \eqref{pfthm:RSHcurves1} and Lemma \ref{lem:Q1-G} the map $\varphi$ is also surjective.
Altogether we see that $\varphi$ is an isomorphism which clearly fits into the following
commutative diagram
\[\xymatrix{ G(\sC,D)\ar[r]^{\Div}\ar[d]^\varphi_{\simeq} & {\rm Div}(\sC,D)\ar@{=}[d]^{\text{Lem \ref{lem:Div-C}}}\\
                   NQ_1/\im(\delta_2)\ar[r]^{\delta_1} & NQ_0.
}\]
This finishes the proof.
\end{proof}

\bibliographystyle{alpha}

\end{document}